\documentclass[DIV=14,abstracton]{scrartcl}
\usepackage[american]{babel}
\usepackage[utf8]{inputenc}
\usepackage[T1]{fontenc}
\usepackage{lmodern,microtype}
\usepackage{amsmath, amssymb, amsthm, bm}
\usepackage{caption, graphicx, pgfplots, subcaption, tikz}

\usepackage{color, lineno} 

\usepackage{hyperref}

\newcommand{\DataPath}{.}
\usetikzlibrary{plotmarks}
\pgfplotsset{compat=newest}
\pgfplotsset{plot coordinates/math parser=false}
\newcommand{\MyLineWidth}{0.65pt}
\newcommand{\MyLineWidthDotted}{1.0pt}
\newcommand{\MyMarkSize}{1.75pt}
\newcommand{\MyColorExact}{black}
\newcommand{\MyColorEsti}{blue}  
\newcommand{\MyColorAlert}{red}

\hypersetup{
    pdftitle   = {Conditioning of Finite Element Equations with Arbitrary Anisotropic Meshes},
    pdfauthor  = {L. Kamenski, W. Huang, H. Xu},
    pdfsubject = {Revision 2}
}

\allowdisplaybreaks

\DeclareMathOperator{\spanM}{span}
\DeclareMathOperator{\tr}{tr}
\providecommand{\abs}[1]{\lvert#1\rvert}
\providecommand{\Abs}[1]{\left\lvert#1\right\rvert}
\providecommand{\norm}[1]{\lVert#1\rVert}
\providecommand{\Norm}[1]{\left\lVert#1\right\rVert}
\providecommand{\dx}{\ d\boldsymbol{x}}
\providecommand{\dxi}{\ d\boldsymbol{\xi}}
\providecommand{\FDF}{(F_K')^{-1} \D_K (F_K')^{-T}}
\newcommand{\D}{\mathbb{D}}
\newcommand{\cO}{\mathcal{O}}
\newcommand{\R}{\mathbb{R}}

\newcommand{\cT}{\mathcal{T}}
\newcommand{\V}[1]{\boldsymbol{#1}}

\newtheorem{theorem}{Theorem}[section]
\newtheorem{lemma}[theorem]{Lemma}

\theoremstyle{definition}
\newtheorem{example}[theorem]{Example}
\newtheorem{scase}[theorem]{Special Case}

\theoremstyle{remark}
\newtheorem{remark}[theorem]{Remark}

\begin{document}

\title{Conditioning of Finite Element Equations\\
   with Arbitrary Anisotropic Meshes}

\author{Lennard Kamenski\thanks{Department of Mathematics, the University of Kansas,
   Lawrence, KS 66045 (\href{mailto:lkamenski@math.ku.edu}{\nolinkurl{lkamenski@math.ku.edu}})},
\and Weizhang Huang\thanks{Department of Mathematics, the University of Kansas,
   Lawrence, KS 66045 (\href{mailto:huang@math.ku.edu}{\nolinkurl{huang@math.ku.edu}})},
\and Hongguo Xu\thanks{Department of Mathematics, the University of Kansas,
   Lawrence, KS 66045 (\href{mailto:xu@math.ku.edu}{\nolinkurl{xu@math.ku.edu}})}
}

\begin{titlepage}
\maketitle
\begin{abstract}
Bounds are developed for the condition number of the linear finite element equations of an anisotropic diffusion problem with arbitrary meshes. 
They depend on three factors.
The first, factor proportional to a power of the number of mesh elements, represents the condition number of the linear finite element equations for the Laplacian operator on a uniform mesh.
The other two factors arise from the mesh nonuniformity viewed in the Euclidean metric and in the metric defined by the diffusion matrix.
The new bounds reveal that the conditioning of the finite element equations with adaptive anisotropic meshes is much better than what is commonly feared.
Diagonal scaling for the linear system and its effects on the conditioning are also studied.
It is shown that the Jacobi preconditioning, which is an optimal diagonal scaling for a symmetric positive definite sparse matrix, can eliminate the effects of mesh nonuniformity viewed in the Euclidean metric and reduce those effects of the mesh viewed in the metric defined by the diffusion matrix.
Tight bounds on the extreme eigenvalues of the stiffness and mass matrices are obtained.
Numerical examples are given.
\end{abstract}
\begin{description}
\item[Keywords:] mesh adaptation, anisotropic mesh, finite element,
      mass matrix, stiffness matrix, conditioning, extreme eigenvalues,
      preconditioning, diagonal scaling
\item[2010 MSC:] 65N30, 65N50, 65F35, 65F15
\end{description}
\end{titlepage}

\section{Introduction}
\label{sect:introduction}

It has been amply demonstrated that significant improvements in accuracy can be gained when an appropriately chosen anisotropic mesh is used for the numerical solution of problems exhibiting anisotropic features.
However, there exists a general concern in the scientific computing community that an anisotropic mesh, which can contain elements of large aspect ratio and small volume, may lead to ill-conditioned linear systems and this could outweigh the accuracy and efficiency improvements gained by anisotropic mesh adaptation.
For isotropic mesh adaptation, Bank and Scott~\cite{BanSco89} (also see Brenner and Scott~\cite{BreSco08}) show that after proper diagonal scaling, the condition number of finite element equations with an adaptive mesh is essentially the same as that for a uniform mesh.
Unfortunately, this result does not apply to anisotropic meshes nor to problems with anisotropic diffusion.

For problems with anisotropic diffusion and arbitrary meshes, several estimates have been developed for the extreme eigenvalues of the stiffness matrix.
For example, Fried~\cite{Fried73} shows that the largest eigenvalue of the stiffness matrix is bounded by the largest eigenvalues of element stiffness matrices.
Shewchuk~\cite{Shewch02} obtains sharp  bounds on largest eigenvalues of element stiffness matrices for linear triangular and tetrahedral finite elements.
More recently, Du et al.~\cite{DuWaZh09} develop a bound that can be viewed as a generalization of Shewchuk's result to general dimensions and simplicial finite elements.

Estimation of the smallest eigenvalue for the general case appears to be more challenging.
Standard estimates (e.g., see Ern and Guermond~\cite{Ern04}) are linearly proportional to the volume of the smallest mesh element, which is typically too pessimistic for nonuniform meshes.
Moreover, Apel~\cite[Sect.~4.3.3]{Apel99} shows that the order of the smallest eigenvalue of the stiffness matrix for a specific, specially designed anisotropic mesh is the same as for a uniform mesh.
As a matter of fact, coefficient adaptive anisotropic meshes can even improve the conditioning for partial differential equations (PDEs) with anisotropic diffusion coefficients, as observed by D'Azavedo et al.~\cite{DARoDo97} and Shewchuk~\cite[Sect.~3.2]{Shewch02}.
A noticeable approach for obtaining sharper bounds for the smallest eigenvalue is proposed by Fried~\cite{Fried73}.
The approach employs a continuous generalized eigenvalue problem with an auxiliary density function and its key is to find a lower bound for the smallest eigenvalue of the continuous problem.
Bounds for the smallest eigenvalue of the stiffness matrix obtained with Fried's approach are valid for general meshes in any dimension but in $d \geq 3$ dimensions they are less sharp than those obtained in this paper.

The objective of this paper is threefold.
First, we develop tight bounds on the extreme eigenvalues and the condition number of the stiffness matrix for a general diffusion problem with an arbitrary anisotropic mesh.
No assumption on the shape or size of mesh elements is made in the development.
Our upper bound on the largest eigenvalue can be also expressed in terms of mesh nonuniformity viewed in the metric tensor defined by the diffusion matrix (which will hereafter be referred to as \emph{the mesh $\D$-nonuniformity}). 
It is comparable to those of Shewchuk~\cite{Shewch02} and Du et al.~\cite{DuWaZh09} but is expressed as a sum of patch-wise terms instead of element-wise terms as in the aforementioned references.
The patch-wise nature gives a sharper bound and makes it more convenient to use in the development of diagonal scaling preconditioners.
To obtain lower bounds for the smallest eigenvalue of the stiffness matrix we extend Bank and Scott's result~\cite{BanSco89} to arbitrary meshes.
This generalization is not trivial and special effort has to be made to deal with the arbitrariness of the mesh.
Along the way we establish anisotropic upper and lower bounds on the extreme eigenvalues of the mass matrix which are much tighter than estimates available in the literature.

The second objective of the paper is to provide a clear geometric interpretation for the obtained bounds on the condition number of the stiffness matrix. 
These bounds are shown to depend on three factors.
The first factor is proportional to a power of the number of mesh elements and represents the condition number of the stiffness matrix for the linear finite element approximation of the Laplacian operator on a uniform mesh.
The other two factors arise from the mesh nonuniformity in volume measured in the Euclidean metric (which will be referred to \emph{the mesh volume-nonuniformity}) and from the mesh $\D$-nonuniformity.

The third objective is to study diagonal scaling for the finite element linear system and its effects on the conditioning.
We focus on the scaling with the diagonal entries of the matrix (Jacobi preconditioning) since it is an optimal diagonal scaling for a symmetric positive definite sparse matrix~\cite[Corollary~7.6 and the following]{Higham96}.
We show that the Jacobi preconditioning can eliminate the effects of the mesh volume-nonuniformity and improve those caused by the mesh $\D$-nonuniformity, thus significantly reducing the effects of the mesh irregularity on the conditioning.
From the practical point of view, this result indicates that a simple diagonal preconditioning can effectively transform the stiffness matrix into a matrix which has a comparable condition number as the one with a uniform mesh.

The outline of the paper is as follows.
Section~\ref{sect:discretization} briefly describes a linear finite element discretization of a general anisotropic diffusion problem.
Estimation of the extreme eigenvalues and the condition number of the mass matrix is given in Sect.~\ref{sect:massMatrix}.
Section~\ref{sect:lambdaMax} deals with the estimation of the largest eigenvalue of the stiffness matrix.
Bounds on the smallest eigenvalue and the condition number of the stiffness matrix and the effects of diagonal scaling are investigated in Sect.~\ref{sect:lambdaMinSobolev}.
A selection of examples in Sect.~\ref{sect:numericalExperiments} provides a numerical validation for the theoretical findings.
Finally, conclusions and further remarks are given in Sect.~\ref{sect:summary}. 

\section{Linear finite element approximation}
\label{sect:discretization}
We consider the boundary value problem (BVP) of a general diffusion differential equation in the form 
\begin{equation}
   \begin{cases}
      - \nabla \cdot \left( \D \nabla u \right) = f, & \text{in $\Omega$}, \\
      u = 0,                            & \text{on $\partial\Omega$},
   \end{cases}
   \label{eq:bvp1}
\end{equation}
where $\Omega$ is a simply connected polygonal or polyhedral domain in $\mathbb{R}^d$ ($d \geq 1$) and $\D = \D (\V{x})$ is the diffusion matrix. 
We assume that $\D$ is symmetric and positive definite and there exist two positive constants $d_{\min}$ and $d_{\max}$ such that
\begin{equation}
   d_{\min} I \leq \D(\V{x}) \leq d_{\max} I, \quad \forall \V{x} \in \Omega,
   \label{eq:D:1}
\end{equation}
where the less-than-or-equal sign means that the difference between the right-hand side and left-hand side terms is positive semidefinite.

We are interested in the linear finite element solution of BVP  \eqref{eq:bvp1}. 
Assume that an affine family $\lbrace \cT_h \rbrace$ of simplicial decompositions  of $\Omega$ is given and denote the associated linear finite element space by $V^h \subset H_0^1(\Omega)$. 
A linear finite element solution $u_h \in V^h$ to BVP \eqref{eq:bvp1} is defined by
\[
   \int_{\Omega} \nabla v_h \cdot \D \nabla u_h \dx
      = \int_{\Omega} f v_h \dx, \quad \forall v_h \in V^h ,
\]
or
\[
   \sum_{K \in \cT_h} \int_{K} \nabla v_h \cdot  \D \nabla u_h \dx 
      = \sum_{K \in \cT_h}  \int_{K} f v_h \dx , \quad \forall v_h \in V^h.
\]
Since both $\nabla u_h$ and $\nabla v_h$ are constant on $K$, we can rewrite the above equation as
\begin{equation}
   \sum_{K \in \cT_h} \Abs{K} \nabla v_h \cdot  \D_K  \nabla u_h 
   = \sum_{K \in \cT_h}  \int_{K} f  v_h \dx, 
      \quad \forall v_h \in V^h,
   \label{fem-3}
\end{equation}
where $\D_K$ is the integral average of $\D$ over $K$, i.e.,
\begin{equation}
   \D_K = \frac{1}{\Abs{K}} \int_K \D(\V{x}) \dx .
   \label{eq:D:2}
\end{equation}
In practice, the integrals in \eqref{fem-3} and \eqref{eq:D:2} have to be approximated numerically via a quadrature rule.
Although this will change the definition of $\D_K$ and the right-hand side term of \eqref{fem-3} slightly, the procedure and the results in this paper will remain valid for this situation.

Finite element equation \eqref{fem-3} can be expressed in a matrix form.
Denoting the numbers of elements and interior vertices of $\cT_h$ by $N$ and $N_{vi}$ and assuming that the vertices are ordered in such a way that the first $N_{vi}$ vertices are the interior vertices, we have
\begin{align}
   V^h &= \spanM \lbrace \phi_1, \dotsc, \phi_{N_{vi}} \rbrace, \nonumber \\
   u_h &= \sum_j u_j \phi_j ,
   \label{uh-1}
\end{align}
where $\phi_j$ is the linear basis function associated with the $j^{\text{th}}$ vertex.
In \eqref{uh-1} and hereafter, we use the sum $\sum_j$ with the index $j$ ranging over all interior vertices, i.e.,
\[
\sum_j = \sum_{j=1}^{N_{vi}} 
\]

Substituting \eqref{uh-1} into \eqref{fem-3} and taking $v_h = \phi_i$ ($i=1,\dotsc, N_{vi}$), we obtain the linear algebraic system
\[
   A \V{u} = \V{f},
\]
where $\V{u} = [u_1,\dotsc, u_{N_{vi}}]^T$ and the stiffness matrix $A$ and the right-hand side term $\V{f}$ are given by
\begin{align}
   A_{ij} &= \sum_{K \in \cT_h} \Abs{K} \nabla\phi_i\vert_K \cdot \D_K \nabla \phi_j \vert_K,
      & i, j &= 1,\dotsc, N_{vi},
   \label{eq:Aij} \\ 
   f_i &= \sum_{K \in \cT_h}  \int_{K} f \phi_i \dx, 
      & i &= 1,\dotsc, N_{vi} ,
      \notag
\end{align}
and $\nabla \phi_i \vert_K$ and $\nabla \phi_j \vert_K$ denote the restriction of $\nabla \phi_i$ and $\nabla \phi_j$ on $K$.
Our main goal is to estimate the condition number of stiffness matrix $A$.

\section{Mass matrix}
\label{sect:massMatrix}

To start with, we consider the element mass matrix $\hat B$ for the reference element $\hat{K}$ (which is assumed to be unitary, i.e., $\abs{\hat{K}} = 1$),
\[
   \hat B = (\hat b_{i,j} ),
   \quad \hat b_{i,j} = \int_{\hat K} \hat \phi_i  \hat \phi_j \dxi,
   \quad i,j = 1,\dotsc, d+1,
\]
where $\hat \phi_i$'s are the linear basis functions associated with the vertices of $\hat K$.
Matrix $\hat B$ is symmetric and positive definite.
Moreover, for the standard linear finite elements considered in this paper we have
\[ 
   \frac{1}{(d+1)(d+2)} I \leq \hat B \leq \frac{1}{d+1} I.
\]

\subsection{Condition number of the mass matrix}
\label{SEC:mass-matrix}

Consider the global mass matrix
\[
    B = (B_{i,j}), 
    \quad B_{i,j} = \int_\Omega \phi_i \phi_j \dx,
    \quad i, j  = 1,\dotsc, N_{vi}.
\]
Notice that
\begin{equation}
   B_{jj} 
      = \int_\Omega \phi_j^2 \dx 
      = \sum_{K \in \omega_j} \int_K \phi_j^2 \dx
      = \sum_{K \in \omega_j} \frac{2 \Abs{K}}{(d+1)(d+2)} 
      = \frac{2 \Abs{\omega_j}}{(d+1)(d+2)},
   \label{eq:B:diag:1}
\end{equation}
where $\omega_j$ is the element patch associated with the $j^{\text{th}}$ vertex and $\Abs{\omega_j}$  is its volume.

The following theorem gives lower and upper bounds on the condition number of the mass matrix for any dimension and any mesh.

\begin{theorem}[Condition number of the mass matrix]
\label{thm:conditionMM}
The condition number of the mass matrix for the linear finite elements on a simplicial mesh is bounded by
\begin{equation}
   \frac{\max_j B_{jj}}{ \min_j B_{jj}}
   \le \kappa(B) \le 
   (d+2) \frac{\max_j B_{jj}}{ \min_j B_{jj}} .
   \label{eq:massMatrixBound:1}
\end{equation}
\end{theorem}

\begin{proof}

For an element $K$, let $\V{u}_K$ be the restriction of the vector $\V{u}$ on $K$ and $B_K$ the element mass matrix.
Then, 
\[
   \V{u}^T B \V{u}
   = \sum_{K\in\cT_h} \V{u}_K^T B_K \V{u}_K
   = \sum_{K\in\cT_h} \Abs{K} \V{u}_K^T \hat B \V{u}_K 
   \leq \frac{1}{d+1} \sum_{K\in\cT_h} \Abs{K} \Norm{\V{u}_K}_2^2 .
\] 
Rearranging the sum on the right-hand side according to the vertices and using \eqref{eq:B:diag:1},
\[
   \V{u}^T B \V{u}  
      \le \frac{1}{d+1} \sum_{K\in\cT_h} \Abs{K} \Norm{\V{u}_K}_2^2 
      = \frac{1}{d+1} \sum_{j} u_j^2 \Abs{\omega_j}
      = \frac{d+2}{2} \sum_{j} u_j^2 B_{jj},
\]
which implies
\[
\lambda_{\max}(B) \le \frac{d+2}{2} \max_{j} B_{jj} .
\]
Similarly, we have
\[
   \V{u}^T B \V{u} \geq  \frac{1}{2} \sum_{j} u_j^2 B_{jj}
   \quad \text{and}\quad 
   \lambda_{\min}(B) \geq \frac{1}{2} \min_j B_{jj} .
\]
Moreover, it is easy to show that
\[
   \lambda_{\max}(B) \geq \max_j B_{jj}
   \quad \text{and} \quad
   \lambda_{\min}(B) \leq \min_j B_{jj}.
\]
Combining the above estimates gives
\begin{align*}
   \max_j B_{jj}  &\le \lambda_{\max}(B)   \le \frac{d+2}{2} \max_j B_{jj}, \\
   \frac{1}{2} \min_j B_{jj}   &\le \lambda_{\min}(B)  \leq  \min_j B_{jj},
\end{align*}
from which estimate \eqref{eq:massMatrixBound} follows.
\end{proof}

From \eqref{eq:B:diag:1}, Theorem~\ref{thm:conditionMM} implies  
\begin{equation}
   \frac{\Abs{\omega_{\max}}}{\Abs{\omega_{\min}}} 
   \leq \kappa(B)
   \leq (d+2) \frac{\Abs{\omega_{\max}}}{\Abs{\omega_{\min}}},
   \label{eq:massMatrixBound}
\end{equation}
where $\Abs{\omega_{\max}} = \max_j \Abs{\omega_j}$ and $\Abs{\omega_{\min}} = \min_j \Abs{\omega_j}$.
If the mesh is quasi-uniform, this means ${\Abs{\omega_{\max}}}/{\Abs{\omega_{\min}}} = \cO(1)$ and $\kappa(B) = O(1)$.
On the other hand, ${\Abs{\omega_{\max}}}/{\Abs{\omega_{\min}}}$  and $\kappa(B)$ can become large
for nonuniform meshes.

\subsection{Relation to the estimates in the literature}
\label{rem:massStandardEstimates}
If we denote the maximum number of mesh elements in a patch by $p_{\max}$, then
\[
   \Abs{K_{\min}} \leq \Abs{\omega_j} \leq p_{\max} \Abs{K_{\max}}, \quad \forall j = 1,\dotsc, N_{vi} 
\]
and estimate \eqref{eq:massMatrixBound} implies
\begin{equation}
   \kappa(B) \leq (d+2) p_{\max} \frac{\Abs{K_{\max}}}{\Abs{K_{\min}}},
   \label{eq:kappaB:Fried}
\end{equation}
which is the bound obtained by Fried~\cite[inequality~(24)]{Fried73}.

Moreover, for an isotropic mesh,
\[
   \Abs{K_{\max}} \propto h_{\max}^d , \quad \Abs{K_{\min}} \propto h_{\min}^d,
\]
where $h_{\max}$ and $h_{\min}$ are the largest and smallest element diameters.
Substituting this into bound \eqref{eq:kappaB:Fried} gives
\begin{equation}
   \kappa(B) \leq C \left( \frac{h_{\max}}{h_{\min}} \right)^d,
   \label{eq:kappaB:standard}
\end{equation}
which is precisely the standard estimate found in the literature (e.g.,~\cite[Rem.~9.10]{Ern04}).

For anisotropic meshes, on the other hand, the new estimate \eqref{eq:massMatrixBound} is much tighter than both Fried's estimate \eqref{eq:kappaB:Fried} and the standard estimate \eqref{eq:kappaB:standard}, since large ${\Abs{K_{\max}}}/{\Abs{K_{\min}}}$ and ${h_{\max}}/{h_{\min}}$ do not necessarily imply large ${\Abs{\omega_{\max}}} / {\Abs{\omega_{\min}}}$ for those meshes.\footnote{For example, meshes in Fig.~\ref{fig:exampleN} have ${\Abs{K_{\max}}}/{\Abs{K_{\min}}} \to \infty$ but ${\Abs{\omega_{\max}}}/{\Abs{\omega_{\min}}} = \cO(1)$.}
Furthermore, estimate \eqref{eq:massMatrixBound} also provides a tight lower bound, which is not available with \eqref{eq:kappaB:Fried} and \eqref{eq:kappaB:standard}.

\subsection{Diagonal scaling for the mass matrix}
\label{ssect:massMatrixScaled}

It is known~\cite[Corollary~7.6 and the following]{Higham96} that for a symmetric positive definite sparse matrix, scaling by its diagonal entries (Jacobi preconditioning) is an optimal diagonal preconditioning (up to a constant depending on the maximum number of non-zeroes per column and row of the matrix).
We are interested in a bound on the condition number after such preconditioning.

For a diagonal scaling $S = (s_j)$, similarly to Theorem~\ref{thm:conditionMM} we obtain
\[
\frac{\max_j s^{-2}_j B_{jj}}{\min_j s^{-2}_j B_{jj}} 
      \leq \kappa(S^{-1}BS^{-1}) \leq (d+2) \frac{\max_j s^{-2}_j B_{jj}}{\min_j s^{-2}_j B_{jj}}
\]
and, for the Jacobi preconditioning $s_j^2 = B_{jj}$, we have arrived at the following theorem by Wathen~\cite{Wathen87} who studies the effects of the diagonal scaling on the condition number of the Galerkin mass matrix.
\begin{theorem}[{\cite[Table~1]{Wathen87}}]
\label{thm:scaledMM}
The condition number of the Jacobi preconditioned Galerkin mass matrix with a simplicial mesh has a mesh-independent bound\[
   \kappa (S^{-1}BS^{-1}) \leq d + 2.
\]
\end{theorem}

Theorems~\ref{thm:conditionMM} and~\ref{thm:scaledMM} show that \emph{the mesh volume-nonuniformity has a significant effect} on the condition number of the mass matrix and \emph{this effect is completely eliminated} by the Jacobi preconditioning.

As we will see later in Sect.~\ref{sect:lambdaMinSobolev}, diagonal scaling plays a similar role in reducing the effects of mesh nonuniformity on the condition number of the stiffness matrix.

\section{Largest eigenvalue of the stiffness matrix}
\label{sect:lambdaMax}

The following lemma is valid for any dimension.

\begin{lemma}[Largest eigenvalue]
\label{lem:lambdaMax}
The largest eigenvalue of the stiffness matrix $A = (A_{ij})$ for the linear finite element approximation of BVP \eqref{eq:bvp1} is bounded by
\begin{equation}
   \max\limits_{j} A_{jj} \leq \lambda_{\max}(A)  \leq (d+1) \max\limits_{j} A_{jj}.
   \label{eq:lambdaMaxA}
\end{equation}

The largest eigenvalue of the diagonally (Jacobi) preconditioned stiffness matrix $S^{-1}AS^{-1}$ has a mesh-independent bound
\begin{equation}
   1 \leq \lambda_{\max}(S^{-1}AS^{-1} ) \leq d + 1.
   \label{eq:lambdaMaxScaled}
\end{equation}

\end{lemma}

\begin{proof} 
First, 
recall that for any symmetric positive semidefinite matrix $M$,
\[
   \V{v}^T M \V{w}
      \leq \frac{1}{2}\left( \V{v}^T M \V{v} + \V{w}^T M \V{w} \right), 
   \quad \forall  \V{v}, \V{w} \in \R^{d+1}.
\]
Then, using the local indices on $K$ and the definition of $A_{jj}$ from \eqref{eq:Aij} and rearranging the sum according to the vertices, we have
\begin{align*}
\V{u}^T A \V{u} 
   &= \int_\Omega \nabla u_h \cdot \D \nabla u_h \dx \\
   &= \sum_{K \in \cT_h} \Abs{K} \sum_{i_K, j_K = 1}^{d+1} 
      \left(u_{i_K} \nabla \phi_{j_K} \right) \cdot 
     \D_K \left(u_{j_K} \nabla \phi_{j_K} \right) \\
  &\leq (d+1) \sum_{K \in \cT_h} \Abs{K} \sum_{j_K= 1}^{d+1}
   \left(u_{j_K} \nabla\phi_{j_K}\right) \cdot \D_K \left(u_{j_K} \nabla\phi_{j_K} \right) \\
  &= (d+1) \sum_j u_j^2 \sum_{K \in \omega_j} \Abs{K} \nabla\phi_j \cdot \D_K \nabla\phi_j \\
  &= (d+1) \sum_j u_j^2 A_{jj} \\
  &\leq (d+1)  \norm{\V{u}}_2^2 \max_j A_{jj}.
\end{align*}
On the other hand, using the canonical basis vectors $\V{e_j}$ we have
\[
   \lambda_{\max}(A) \geq \V{e_j}^T A \V{e_j} = A_{jj}, \quad  j = 1,\dotsc,N_{vi},
\]
and altogether we get \eqref{eq:lambdaMaxA}.

Using the same procedure for a diagonal scaling $S = (s_j)$ we obtain
\[
   \max_j (s_j^{-2} A_{jj}) \leq \lambda_{\max}(S^{-1}AS^{-1}) \leq (d+1) \max_j (s_j^{-2} A_{jj}).
\]
For the Jacobi preconditioning we have $s_j^2 = A_{jj}$, which gives estimate \eqref{eq:lambdaMaxScaled}.
\end{proof}

\begin{remark}
Bound \eqref{eq:lambdaMaxScaled} can also be obtained by using the unassembled form of $A$ as shown in~\cite[Sect.~3]{Wathen87}.
However, the analysis employed in~\cite{Wathen87} cannot provide a lower bound on $\lambda_{\min}(S^{-1} A S^{-1})$ other than the trivial one, $\lambda_{\min}(S^{-1} A S^{-1}) \geq 0$.
\end{remark}

\subsection{Geometric interpretation}
\label{ssect:lambdaMaxGeom}
Although Lemma~\ref{lem:lambdaMax} gives a very tight bound on $\lambda_{\max}(A)$, it does not provide any explanation on how the mesh or the diffusion matrix affect the conditioning.
We now derive a bound on $\lambda_{\max}(A)$ in terms of mesh quantities and the diffusion matrix.

Let $F_K \colon \hat{K} \to K$ be the affine mapping from the reference element $\hat{K}$ to the mesh element $K$, $F_K'$ the Jacobian matrix of $F_K$, $j_K$ the local index of $\phi_j$ on $K$ and $\hat{\phi}_{j_K} = F_K \circ \phi_{j_K}$ the corresponding basis function on $\hat{K}$.

Using the chain rule, we have
\begin{align}
 A_{jj} 
   & =  \sum_{K \in \omega_j} \Abs{K} \nabla\phi_j \cdot \D_K \nabla\phi_j \notag \\
   &=  \sum_{K \in \omega_j} \Abs{K} 
      \bigl((F_K')^{-T} \nabla \hat{\phi}_{j_K}\bigr)  
         \cdot \D_K \bigl( (F_K')^{-T} \nabla \hat{\phi}_{j_K} \bigr)\notag \\
   &\leq  \sum_{K \in \omega_j} \Abs{K} \Norm{\FDF}_2 \norm{\nabla \hat{\phi}_{j_K}}_2^2 \notag \\
   & \leq C_{\hat{\phi}} \sum_{K\in\omega_j} \Abs{K}  \Norm{\FDF}_2 ,
\label{eq:scaling} 
\end{align}
where $C_{\hat \phi}  = \max_{i_K = 1,\dotsc, d+1}  \norm{\nabla \hat\phi_{i_K}}_2^2$.
Combining this with Lemma~\ref{lem:lambdaMax} yields
\begin{equation}
   \lambda_{\max}(A)  \leq (d+1) C_{\hat{\phi}} \max_j \sum_{K\in\omega_j} \Abs{K}  \Norm{\FDF}_2 .
   \label{eq:lambdaMaxA:1}
\end{equation}
\begin{remark}
If we denote the maximum number of elements meeting at a mesh point by $p_{\max}$, then bound \eqref{eq:lambdaMaxA:1} implies
\[
   \lambda_{\max}(A) 
         \leq p_{\max} (d+1) C_{\hat{\phi}} \max_K  \Bigl(\Abs{K} \Norm{\FDF}_2 \Bigr),
\]
which is comparable to the estimates mostly found in the literature (e.g.,~\cite{DuWaZh09,Fried73,Shewch02}).
Note that both this bound and \eqref{eq:lambdaMaxA:1} are less tight than the bound \eqref{eq:lambdaMaxA}.
\end{remark}

\begin{remark}
Bound \eqref{eq:lambdaMaxA:1} implies that the scaling
\[
   s_j^2 = \sum\limits_{K\in\omega_j}\Abs{K} \Norm{\FDF}_2, \quad j = 1,\dotsc, N_{vi}
\]
will also lead to bounds similar to \eqref{eq:lambdaMaxScaled} and those in Sect.~\ref{sect:lambdaMinSobolev}.
In general, this scaling is greater than the Jacobi preconditioning (cf. \eqref{eq:scaling}) 
although they are equal in 1D or for a mesh that is uniform in the metric specified by $\D^{-1}$ (cf. Sect.~\ref{ssect:geomInt:GeneralCase}).

\end{remark}

\subsubsection{Special case  $\D = I$}
For the simplest case of $\D = I$, bound \eqref{eq:lambdaMaxA:1} has a rather simple interpretation.
The quantity $\norm{(F_K')^{-1}}_2$ can be bounded by the reciprocal of the in-diameter $h_{\min,K}$ of $K$~\cite[Lemma~5.1.2]{Huang11}. 
If we denote the average aspect of $K$ by $\bar{h}_{K}$ (i.e., $\bar{h}_{K} = \Abs{K}^{\frac{1}{d}}$), then we can rewrite \eqref{eq:lambdaMaxA:1} as
\[
  \lambda_{\max}(A) \leq C_{\hat \phi}\max_j \sum_{K \in \omega_j} 
   \Abs{K} \left(\frac{1}{h_{\min,K}}\right)^2 =
   C_{\hat \phi}\max_j \sum_{K \in \omega_j} 
   \Abs{K}^{\frac{d-2}{d}} \left(\frac{\bar{h}_{K}}{h_{\min,K}}\right)^2.
\]
The ratio $\frac{\bar{h}_{K}}{h_{\min,K}}$ is a measure of the \emph{aspect ratio} of $K$.
Thus, for the case of $\D = I$, the largest eigenvalue of $A$ is bounded by the maximum volume-weighted element aspect ratio of the mesh.
This is consistent with the observation by Shewchuk in~\cite{Shewch02} where a detailed discussion on the relation between the largest eigenvalue of the stiffness matrix and the element aspect ratio is available for the case of $\D=I$ in $d=2$ and $d=3$ dimensions.

For the general case $\D \neq I$, on the other hand, it is more convenient to interpret bound \eqref{eq:lambdaMaxA:1} in terms of the \emph{mesh quality measures} introduced in~\cite{Huang05}.
We now proceed with this.

\subsubsection{Mesh quality measures}
\label{ssect:qualityMeasures}
The first measure is the \emph{alignment quality measure}, which can be simply viewed as an equivalent to the aspect ratio of $K$ in the metric specified by $\D_K^{-1}$.
It is defined as
\[
   Q_{ali,\D^{-1}}(K) = {\left(\frac{\frac{1}{d} \tr\bigl( \FDF \bigr)}
   {\det{\bigl( \FDF \bigr)}^{\frac 1 d}} \right)}^{\frac{d}{2(d-1)}}
\]
and measures how closely the principal directions of the circumscribed ellipsoid of $K$ are aligned with the eigenvectors of $\D_K$ and the semi-lengths of the principal axes are proportional to the eigenvalues~\cite{Huang11}.
Notice that 
\[ 
   1 \le Q_{ali,\D^{-1}}(K) < \infty.
\] 
In particular, $Q_{ali,\D^{-1}}(K) = 1$ implies that $K$ is equilateral in the metric $\D_K^{-1}$.

The second measure is the \emph{equidistribution quality measure} defined as the ratio of the average element volume to the volume of $K$, both measured in the metric specified by $\D_K^{-1}$,
\begin{equation}
   Q_{eq,\D^{-1}}(K) = \frac{\frac{1}{N} \sigma_h}{\Abs{K}_{\D_K^{-1}}},
   \label{Q-eq}
\end{equation}
where $\Abs{K}_{\D_K^{-1}} = \Abs{K} \text{det}(\D_K)^{-\frac 1 2}$ is the volume of $K$ with respect to $\D_K^{-1}$ and
\begin{equation}
   \sigma_h = \sum_{K \in \cT_h} \Abs{K}_{\D_K^{-1}} .
   \label{sigma-1}
\end{equation}
The equidistribution quality measure satisfies
\[ 
   0 < Q_{eq,\D^{-1}}(K) < \infty 
   \quad \text{and} \quad
   \frac{1}{N} \sum_{K \in \cT_h} Q_{eq,\D^{-1}}^{-1}(K) = 1.
\] 
Notice that
\[
   \sigma_h = \sum_{K \in \cT_h} \Abs{K} \det(\D_K)^{-\frac 1 2}
      \to \int_\Omega \det\bigl(\D(\V{x})\bigr)^{-\frac 1 2} \dx = \Abs{\Omega}_{\D^{-1}}
\]
as the mesh is being refined.
As a consequence, $\sigma_h$ can be considered as a constant.

\subsubsection{Geometric interpretation (general case)}
\label{ssect:geomInt:GeneralCase}
Using the quality measures we can rewrite the key factor $\norm{\FDF}_2$  as
\begin{align*}
  \Norm{\FDF}_2 
   & \le  \tr \left( \FDF \right) \notag \\
   & = d Q_{ali,\D^{-1}}^{\frac{2(d-1)}{d}}(K) 
      \Bigl( \Abs{K} \det (\D_K)^{-\frac{1}{2}} \Bigr)^{- \frac 2 d } \notag \\
   & = d {\left( \frac{N}{\sigma_h} \right)}^{\frac 2 d}
      \left [ \frac{}{} Q_{ali,\D^{-1}}^{d-1}(K) Q_{eq,\D^{-1}}(K) \right ]^{\frac 2 d}
\end{align*}
and therefore
\begin{equation}
   \lambda_{\max}(A) 
      \leq C {\left( \frac{N}{\sigma_h} \right)}^{\frac{2}{d}}  
         \max_j \sum_{K\in\omega_j} 
         \Abs{K} \left[ \frac{}{}Q_{ali,\D^{-1}}^{d-1}(K) 
            Q_{eq,\D^{-1}}(K) \right]^{\frac 2 d}.
   \label{eq:lambdaMax:QualityEstimate}
\end{equation}
Thus, $\lambda_{\max}(A)$ is bounded by the maximum volume-weighted, combined alignment and equidistribution measure of the mesh in the metric $\D_K^{-1}$.

When a mesh is adapted to the coefficients of the BVP, i.e., it is uniform in the metric $\D^{-1}$, it will have the properties
\begin{equation}
   Q_{ali,\D^{-1}}(K) = 1, 
      \quad Q_{eq,\D^{-1}}(K) = 1,  \quad \forall  K \in \cT_h
   \label{eq:M:u:mesh}
\end{equation}
and
\begin{equation}
   {\left( \frac{N}{\sigma_h} \right)}^{\frac 2 d}  \le 
   \Norm{\FDF}_2 \le d {\left( \frac{N}{\sigma_h} \right)}^{\frac 2 d} .
   \label{D-factor-2}
\end{equation}
Moreover,  bound \eqref{eq:lambdaMax:QualityEstimate} will reduce to
\[
   \lambda_{\max}(A) \leq C N^{\frac{2}{d}} \Abs{\omega_{\max}}.
\]

\section{Smallest eigenvalue and condition number of the stiffness matrix}
\label{sect:lambdaMinSobolev}

The approach employed in this section was originally developed by Bank and Scott~\cite{BanSco89} for isotropic meshes.
We generalize it here to arbitrary anisotropic meshes. 

Hereafter, we will use $C$ as a generic constant which can have different values at different appearances but is independent of the mesh, the number of mesh elements, and the solution of the BVP.

We start with bounds on $\lambda_{\min}(A)$.

\begin{lemma}[Smallest eigenvalue]
\label{lem:lambdaMin1d}

The smallest eigenvalue of the stiffness matrix for the linear finite element approximation of BVP \eqref{eq:bvp1} is bounded from below by
\begin{align}
\lambda_{\min}(A) 
   \geq C d_{\min}  N^{-1}  
   \begin{cases}
      1, & \text{for $d = 1$}, \\
      \left(1 +  \ln \frac{\abs{\bar{K}}}{\Abs{K_{\min}}} \right)^{-1}, & \text{for $d = 2$}, \\
      \left(\frac{1}{N}\sum\limits_{K\in\cT_h} 
         {\left(\frac{\abs{\bar{K}}}{\Abs{K}}\right)}^{\frac{d-2}{2}} \right)^{-\frac{2}{d}},
         & \text{for $d\geq3$},
   \end{cases}
   \label{eq:lambdaMin}
\end{align}
where $\abs{\bar{K}} = \frac{1}{N} \Abs{\Omega}$ denotes the average element size.

The smallest eigenvalue of the diagonally (Jacobi) preconditioned stiffness matrix is bounded from below by
\begin{equation}
   \lambda_{\min}(S^{-1}AS^{-1}) 
   \geq C N^{-2} 
      \left( \frac{1}{N d_{\min}} \sum_{K\in\cT_h} \D_K \frac{\abs{\bar{K}}}{\Abs{K}} \right)^{-1},
      \quad \text{for $d = 1$}
   \label{eq:lambdaMinScaled1}
\end{equation}
and
\begin{multline}
 \lambda_{\min}(S^{-1}AS^{-1}) 
   \geq C N^{-\frac{2}{d}} 
      \left(\frac{1}{N d_{\min}^{\frac d 2}} \sum\limits_{K\in\cT_h} \Abs{K} 
       \Norm{\FDF}_2^{\frac{d}{2}} \right)^{-\frac{2}{d}} 
      \\
   \times 
   \begin{cases}
      \left(1 + \biggl|\ln\frac{\max\limits_{K\in\cT_h} \Norm{\FDF}_2}
         {\sum\limits_{K\in\cT_h} \Abs{K} 
             \Norm{\FDF}_2}\biggl|\right)^{-1} ,
         &  \text{for $d = 2$},  \\
      1, &  \text{for $d\geq3$}.
   \end{cases}
   \label{eq:lambdaMinScaled}
\end{multline}

\end{lemma}

\begin{proof}
Since Sobolev's inequality is different for $d=1$, $d=2$ and $d\geq 3$ dimensions~\cite[Theorem~7.10]{Gilbarg01}, we treat these cases separately.

\vspace{0.5em}
Case $d=1$:
Let $C_S$ be the constant associated with Sobolev's inequality.
Using the inequality \eqref{eq:D:1}, Sobolev's inequality, and the equivalence of the vector norms,
\begin{align*} \V{u}^T A \V{u} 
      &\geq d_{\min} \Abs{u_h}^2_{H^1(\Omega)} \\
      &\geq d_{\min} C_S \Abs{\Omega}^{-1} \sup_\Omega \Abs{u_h}^2 \\
      &= d_{\min} C_S \Abs{\Omega}^{-1}  \max_j u_j^2 \\
      &\geq d_{\min} C_S \Abs{\Omega}^{-1}  N^{-1} \Norm{\V{u}}_2^2
\end{align*}
Therefore, $\lambda_{\min}(A) \geq C d_{\min} N^{-1}$.

With scaling,
\begin{equation}
   \V{u}^T S^{-1} A S^{-1}\V{u} 
   \geq C d_{\min} \max_j s_j^{-2} u_j^2 
   \geq C d_{\min}  \frac{\sum_j s_j^2 s_j^{-2} u_j^2}{\sum_j s_j^2} 
   = C d_{\min} \frac{ \Norm{\V{u}}_2^2}{\sum_j s_j^2}.
   \label{eq:lMinProof:1d}
\end{equation}
In 1D, $\nabla \phi_{j_K} \vert_K = \Abs{K}^{-1}$ and therefore
\[
   s_j^2 =A_{jj} 
         = \sum_{K \in \omega_j} \Abs{K} \nabla\phi_j \cdot \D_K \nabla\phi_j 
        = \sum_{K \in \omega_j} \frac{\D_K}{\Abs{K}}.
\]
Using this in  \eqref{eq:lMinProof:1d} gives \eqref{eq:lambdaMinScaled1}.

\vspace{0.5em}
Case $d=2$:
Consider a set of not-all-zero non-negative numbers $\{\alpha_K, K\in \cT_h\}$ (to be determined later) and a finite number $q > 2$.
Let $C_P$, $C_S$, and $C_{\hat{K}}$ be the constants associated with Poincar\'e's inequality, Sobolev's inequality,  and the norm equivalence on $\hat{K}$, respectively.
Using \eqref{eq:D:1}, Poincar\'e's, Sobolev's and Hölder's inequalities and the norm equivalence for $\hat{u}_h$, we have
\begin{align*}
  \V{u}^T A \V{u} & = \int_\Omega \nabla u_h \cdot \D \nabla u_h \dx 
       \geq d_{\min} \Abs{\nabla u_h}_{H^1(\Omega)} \\
      &\geq \frac{d_{\min} C_P}{1+C_P} \Norm{u_h}^2_{H^1(\Omega)}\\
      &\geq \frac{d_{\min} C_P C_S}{1+C_P} \frac{1}{q} \Norm{u_h}^2_{L^q(\Omega)}\\
      & = \frac{d_{\min} C_P C_S}{1+C_P} \frac{1}{q} 
         {\left( \sum_{K \in \cT_h} \Norm{u_h}^q_{L^q(K)} \right)}^{\frac{2}{q}} \\
      & = \frac{d_{\min} C_P C_S}{1+C_P} \frac{1}{q} \left( 
         \sum\limits_{K\in\cT_h} \alpha_K^{\frac{q}{q-2}} \right)^{-\frac{q-2}{q}} 
      {\left( \sum\limits_{K\in\cT_h} \alpha_K^{\frac{q}{q-2}} \right)}^{\frac{q-2}{q}}
      {\left( \sum\limits_{K \in \cT_h} 
         \Norm{u_h}^q_{L^q(K)} \right)}^{\frac{2}{q}} \\
      &\geq \frac{d_{\min} C_P C_S}{1+C_P} \frac{1}{q} \left( 
         \sum\limits_{K\in\cT_h} \alpha_K^{\frac{q}{q-2}} \right)^{-\frac{q-2}{q}}
         \sum_{K\in\cT_h} \alpha_K  \Norm{u_h}^2_{L^q(K)} \\
      &= \frac{d_{\min} C_P C_S}{1+C_P} \frac{1}{q} \left( 
         \sum\limits_{K\in\cT_h} \alpha_K^{\frac{q}{q-2}} \right)^{-\frac{q-2}{q}}
         \sum_{K\in\cT_h} \alpha_K \Abs{K}^{\frac{2}{q}}
            \Norm{\hat{u}_h}^2_{L^q(\hat{K})} \\
      &\geq \frac{d_{\min} C_P C_S C_{\hat{K}}}{1+C_P} \frac{1}{q} \left( 
         \sum\limits_{K\in\cT_h} \alpha_K^{\frac{q}{q-2}} \right)^{-\frac{q-2}{q}}
         \sum_{K\in\cT_h} \alpha_K \Abs{K}^{\frac{2}{q}} \Norm{\V{u}_K}_2^2  \\
      &= \frac{d_{\min} C_P C_S C_{\hat{K}}}{1+C_P} \frac{1}{q} \left( 
         \sum\limits_{K\in\cT_h} \alpha_K^{\frac{q}{q-2}} \right)^{-\frac{q-2}{q}}
         \sum_j u_j^2 \sum_{K\in\omega_j} \alpha_K \Abs{K}^{\frac{2}{q}}.
\end{align*}
The choice $\alpha_K = \Abs{K}^{-\frac{2}{q}}$ yields
\[
   \V{u}^T A \V{u} 
   \geq C d_{\min} \frac{1}{q} 
      \left(\sum\limits_{K\in\cT_h} \Abs{K}^{-\frac{2}{q-2}} \right)^{-\frac{q-2}{q}}
      \sum_j u_j^2
\]
and therefore
\begin{align}
   \lambda_{\min}(A) 
      &\geq C d_{\min} q^{-1}
         \left( \sum_{K\in\cT_h} \Abs{K}^{-\frac{2}{q-2}} \right)^{-\frac{q-2}{q}} \notag \\
      &\geq C d_{\min} q^{-1}
         \left( N \Abs{K_{\min}}^{-\frac{2}{q-2}} \right)^{-\frac{q-2}{q}} \notag \\
      &= C d_{\min} N^{-1} \left[ q^{-1} {\left(  N \Abs{K_{\min}} \right)}^{\frac{2}{q}} \right].
      \label{eq:lambdaMin2:pelimBound}
\end{align}
The largest lower bound on \eqref{eq:lambdaMin2:pelimBound} is obtained for 
$q = \max\left\lbrace 2, \bigl\lvert \ln \left(N\Abs{K_{\min}}\right) \bigr\rvert \right\rbrace$ with
\[
   q^{-1} {\left(  N \Abs{K_{\min}} \right)}^{\frac{2}{q}} 
      \geq \frac{C}{ 1 + \bigl\lvert \ln(N \Abs{K_{\min}})\bigr\rvert}.
\]
The choice $q=2$ is viewed as the limiting case as $q \rightarrow 2^+$.
Estimate \eqref{eq:lambdaMin} follows from this, \eqref{eq:lambdaMin2:pelimBound} and the definition of the average element size.

With scaling, we have
\[
\V{u}^T S^{-1} A S^{-1} \V{u} \geq C d_{\min} \frac{1}{q} \left( 
         \sum\limits_{K\in\cT_h} \alpha_K^{\frac{q}{q-2}} \right)^{-\frac{q-2}{q}}
         \sum_j u_j^2 s_j^{-2} \sum_{K\in\omega_j} \alpha_K \Abs{K}^{\frac{2}{q}} .
\]
For the Jacobi preconditioning $s_j^2 = A_{jj} = \sum_{K\in\omega_j} \nabla \phi_j \cdot \D_K  \nabla \phi_j$ we choose
\[
   \alpha_K 
   = \Abs{K}^{\frac{q-2}{q}}  \sum_{i_K=1}^{d+1} \nabla \phi_{i_K} \cdot \D_K  \nabla \phi_{i_K}
   = \Abs{K}^{\frac{q-2}{q}}  \sum_{i_K=1}^{d+1}
      \nabla \hat{\phi}_{i_K} \cdot \FDF  \nabla \hat{\phi}_{i_K},
\]
which gives
\[
s_j^{-2} \sum_{K\in\omega_j} \alpha_K \Abs{K}^{\frac{2}{q}}  \geq 1
\]
and
\[
\alpha_K \le  (d+1) C_{\hat{\phi}} \Abs{K}^{\frac{q-2}{q}} \Norm{\FDF}_2 ,
\]
where $C_{\hat{\phi}} = \max_{i_K = 1,\dotsc, d+1} \norm{\nabla \hat \phi_{i_K}}^2$.
With these and choosing the value for the index $q$ in a similar manner as for the case without scaling we obtain \eqref{eq:lambdaMinScaled}.

\vspace{0.5em}
Case $d\geq3$:
This case is very similar to case $d = 2$.
Again, from \eqref{eq:D:1}, Poincar\'e's, Sobolev's and Hölder's inequalities and the norm equivalence for $\hat{u}_h$, we have
\begin{align*}
   \V{u}^T A \V{u} & = \int_\Omega \nabla u_h \cdot \D \nabla u_h \dx 
      \geq d_{\min} \abs{\nabla u_h}_{H^1(\Omega)} \\
      &  \geq \frac{d_{\min} C_P}{1+C_P} \Norm{u_h}^2_{H^1(\Omega)}   \\
      &\geq \frac{d_{\min} C_P C_S}{1+C_P} \Norm{u_h}^2_{L^{\frac{2d}{d-2}}(\Omega)} 
         \\
      & = \frac{d_{\min} C_P C_S}{1+C_P}
       \left( \sum_{K \in \cT_h} \Norm{u_h}^{\frac{2d}{d-2}}_{L^{\frac{2d}{d-2}}(K)} 
         \right)^{\frac{d-2}{d}} \\
      & =\frac{d_{\min} C_P C_S}{1+C_P} \left( 
         \sum\limits_{K\in\cT_h} \alpha_K^{\frac{d}{2}} \right)^{-\frac{2}{d}} 
      \left( \sum\limits_{K\in\cT_h} \alpha_K^{\frac{d}{2}} \right)^{\frac{2}{d}}
      \left( \sum\limits_{K \in \cT_h} 
         \Norm{u_h}^{\frac{2d}{d-2}}_{L^{\frac{2d}{d-2}}(K)} \right)^{\frac{d-2}{d}}\\
      &\geq \frac{d_{\min} C_P C_S}{1+C_P}  \left( 
         \sum\limits_{K\in\cT_h} \alpha_K^{\frac{d}{2}} \right)^{-\frac{2}{d}}
         \sum_{K\in\cT_h} \alpha_K  \Norm{u_h}^2_{L^{\frac{2d}{d-2}}(K)}
         \\
   &= \frac{d_{\min} C_P C_S}{1+C_P} \left( 
         \sum_{K\in\cT_h} \alpha_K^{\frac{d}{2}} \right)^{-\frac{2}{d}}
         \sum_{K\in\cT_h} \alpha_K \Abs{K}^{\frac{d-2}{d}}
            \Norm{\hat{u}_h}^2_{L^{\frac{2d}{d-2}}(\hat{K})}\\
      &\geq \frac{d_{\min} C_P C_S C_{\hat{K}}}{1+C_P}  \left( 
         \sum\limits_{K\in\cT_h} \alpha_K^{\frac{d}{2}} \right)^{-\frac{2}{d}}
         \sum_{K\in\cT_h} \alpha_K \Abs{K}^{\frac{d-2}{d}} \Norm{\V{u}_K}_2^2
        \\
      &=  \frac{d_{\min} C_P C_S C_{\hat{K}}}{1+C_P} \left( 
         \sum\limits_{K\in\cT_h} \alpha^{\frac{d}{2}}_K \right)^{-\frac{2}{d}}
         \sum_j u_j^2 \sum_{K\in\omega_j} \alpha_K \Abs{K}^{\frac{d-2}{d}} .
\end{align*}
The choice $\alpha_K = \Abs{K}^{-\frac{d-2}{d}}$ gives
\[
   \V{u}^T A \V{u} 
   \geq C d_{\min} \left( \sum\limits _{K\in\cT_h}\Abs{K}^{\frac{2-d}{2}}  \right)^{-\frac{2}{d}}
      \sum_j u_j^2 .
\]
Estimate \eqref{eq:lambdaMin} follows from this and the definition of the average element size.

The bound for the scaled stiffness matrix is obtained by choosing 
\[
   \alpha_K = \Abs{K}^{\frac{2}{d}} \sum_{i_K=1}^{d+1} 
      \nabla \hat{\phi}_{i_K} \cdot
         \FDF  \nabla \hat{\phi}_{i_K}.
   \qedhere
\]
\end{proof}


Combining Lemma~\ref{lem:lambdaMax}, estimate \eqref{eq:lambdaMaxA:1} and Lemma~\ref{lem:lambdaMin1d} we obtain upper bounds on the condition number of the stiffness matrix and the scaled stiffness matrix.

\begin{theorem}[Condition number of the stiffness matrix]
\label{thm:conditionNew}
The condition number of the stiffness matrix for the linear finite element approximation of BVP \eqref{eq:bvp1} is bounded by
\begin{equation}
 \kappa(A) \le C N^2 \frac{1}{d_{\min}}
    \max_j \sum\limits_{K\in\omega_j} \D_K \frac{\abs{\bar{K}}}{\Abs{K}},
      \quad \text{for $d = 1$}
\label{eq:conditionNew:1D}
\end{equation}
and
\begin{multline}
   \kappa(A) \leq C N^{\frac{2}{d}} \Biggl(\frac{N^{1-\frac{2}{d}}}{d_{\min}} 
       \max_j \sum\limits_{K\in\omega_j}
       \Abs{K} \Norm{\FDF}_2 \Biggl)  \\ 
   \times 
   \begin{cases}
      1 +  \ln \frac{\abs{\bar{K}}}{\Abs{K_{\min}}}, & \text{for $d = 2$}, \\
      \left(\frac{1}{N}\sum\limits_{K\in\cT_h} 
         \left (\frac{\abs{\bar{K}}}{\Abs{K}}\right )^{\frac{d-2}{2}} \right)^{\frac{2}{d}},
         & \text{for $d\geq3$}.
   \end{cases}
   \label{eq:conditionNew}
\end{multline}

The condition number of the diagonally (Jacobi) preconditioned stiffness matrix is bounded by
\begin{equation}
 \kappa(S^{-1}AS^{-1}) 
   \le C N^2  \frac{1}{N d_{\min}}
    \sum_{K\in\cT_h} \D_K \frac{\abs{\bar{K}}}{\Abs{K}},
      \quad \text{for $d = 1$}
\label{eq:conditionNew:1:1}
\end{equation}
and
\begin{multline}
 \kappa(S^{-1}AS^{-1}) 
   \leq C N^{\frac{2}{d}} 
      \left(\frac{1}{N d_{\min}^{\frac d 2}} \sum\limits_{K\in\cT_h}
      \Abs{K} \Norm{\FDF}_2^{\frac{d}{2}} \right)^{\frac{2}{d}} 
      \\
   \times 
   \begin{cases}
      1 + \Abs{\ln\frac{\max\limits_{K\in\cT_h} \Norm{\FDF}_2}
         {\sum\limits_{K\in\cT_h}
         \Abs{K} \Norm{\FDF}_2}} ,
         & \text{for $d = 2$},  \\
      1, & \text{for $d\geq3$}.
   \end{cases}
   \label{eq:conditionNew:1}
\end{multline}
\end{theorem}

\subsection{Geometric interpretation}
\label{sect:geoInt}

We now study the geometric interpretation of the bounds for the condition number.

\subsubsection{Without scaling}
Bounds \eqref{eq:conditionNew:1D} and \eqref{eq:conditionNew} contain three factors, a base bound $C N^{\frac 2 d}$, a factor reflecting the effects of the mesh nonuniformity measured in the metric $\D^{-1}$ (\emph{mesh $\D$-nonuniformity}), and, if $d\geq2$, a factor reflecting the effects of the mesh nonuniformity in volume measured in the Euclidean metric (\emph{volume-nonuniformity}).

The first factor $N^{\frac{2}{d}}$ corresponds to the condition number of the stiffness matrix for the 
Laplacian operator on a uniform mesh (cf. Special Case~\ref{ex:uniform} below).

The second factor 
\[
\frac{N^{1-\frac{2}{d}}}{d_{\min}} \max_j \sum_{K\in\omega_j}
         \Abs{K} \Norm{\FDF}_2
\]
reflects the effects of the mesh $\D$-nonuniformity and can be understood as a volume-weighted, combined alignment and equidistribution quality measure of the mesh with respect to $\D^{-1}$ (cf. Sect.~\ref{ssect:geomInt:GeneralCase}).

The third factor in \eqref{eq:conditionNew} is 
\[
   \begin{cases}
      1 + \ln \frac{\abs{\bar{K}}}{\Abs{K_{\min}}}, & \text{for $d = 2$}, \\
      \left(\frac{1}{N}\sum\limits_{K\in\cT_h} \left (\frac{\abs{\bar{K}}}{\Abs{K}}\right )^{\frac{d}{2}-1} \right)^{\frac{2}{d}},
         & \text{for $d\geq3$}.
\end{cases}
\]
It measures the effects of the mesh volume-nonuniformity (measured in the Euclidean metric) on the condition number.
Notice that there is no effect in 1D and  in 2D it is minimal.
In $d\geq 3$ dimensions the factor is proportional to the average of ${\Abs{K}}^{-1 + \frac{2}{d}}$ over all elements.
This is a significant improvement in comparison with previously available estimates which are proportional to $\Abs{K_{\min}}^{-1}$ \cite{Ern04} or $\Abs{K_{\min}}^{-1 + \frac{2}{d}}$ \cite{Fried73}.

\subsubsection{With scaling}
Bounds \eqref{eq:conditionNew:1:1} and \eqref{eq:conditionNew:1} for the scaled stiffness matrix have the same base bound as without scaling. 
Hence, diagonal scaling has no effect on the condition number when the mesh is uniform and $\D = I$.

Unlike \eqref{eq:conditionNew}, bounds \eqref{eq:conditionNew:1:1} and \eqref{eq:conditionNew:1} do not have the third factor which involves only the element volume (in comparison to the second factor which couples $(F_K^{'})^{-1}$ with $\D^{-1}$).
In this sense, a properly chosen diagonal scaling can eliminate the effects of the mesh volume-nonuniformity on the condition number. 
Moreover, scaling can also significantly reduce the effects of the mesh $\D$-nonuniformity.
Indeed, the factors in \eqref{eq:conditionNew:1:1} and \eqref{eq:conditionNew:1} that couple $(F_K^{'})^{-1}$ with $\D^{-1}$ are asymptotically the $L^{\max\{1,\frac{d}{2}\}}(\Omega)$ norm of $\norm{\FDF}_2$ whereas the corresponding factors in \eqref{eq:conditionNew:1D} and \eqref{eq:conditionNew} are basically the maximum norm.

Furthermore, the $\D$-related factor in \eqref{eq:conditionNew:1} for $d \geq 2$ can be rewritten in terms of  the alignment quality measure $Q_{ali,\D^{-1}}$ from Sect.~\ref{ssect:qualityMeasures} as
\begin{equation}
   \sum_{K\in\cT_h} \Abs{K} \Norm{\FDF}_2^{\frac{d}{2}}
\le d^{\frac d 2} \sum_{K\in\cT_h} Q_{ali,\D^{-1}}^{d-1}(K) \det(\D_K)^{\frac{1}{2}}.
\label{D-factor-1}
\end{equation}
Thus, the dependence of this $\D$-related factor on the element volume is also mild: both $Q_{ali,\D^{-1}}(K)$ and $\D_K$ (the average of $\D$ over $K$) are invariant under the scaling transformation of $K$.

The following special cases are instructional to understand the interplay of the factors for different types of meshes.

\begin{scase}[Uniform meshes]
\label{ex:uniform}
For a uniform mesh and $\D = I$, bounds \eqref{eq:conditionNew:1D}--\eqref{eq:conditionNew:1} yield
\[
   \kappa(A) \leq C N^{\frac{2}{d}}
   \quad \text{and} \quad
   \kappa(S^{-1} A S^{-1}) \leq C N^{\frac{2}{d}},
\]
which is the base bound.
Hence, the diagonal scaling has no effect on the condition number when the mesh is uniform and $\D = I$.
\end{scase}

\begin{scase}[Isotropic meshes, $\D=I$, $d\geq2$]
\label{BS89-3}
For an isotropic mesh and $\D = I$,
\[
\Abs{K} \sim h_K^{d} \quad \text{and} \quad \Norm{\FDF}_2 \sim h_K^{-2}.
\]
Therefore,
\[
   \frac{1}{N} \sum\limits_{K\in\cT_h} \Abs{K} 
      \Norm{\FDF}_2^{\frac{d}{2}}
   \lesssim \frac{1}{N} \sum\limits_{K\in\cT_h} h_K^d h_K^{-d} = 1
\]
and bound \eqref{eq:conditionNew:1} reduces to
\begin{equation}
   \kappa(S^{-1}AS^{-1})  \le C N^{\frac 2 d}
   \begin{cases}
      1 + \ln \frac{\abs{\bar{K}}}{\Abs{K_{\min}}},
         & \text{for $d = 2$,} \\
      1, & \text{for $d \geq 3$,}
   \end{cases}
   \label{BS89:1} 
\end{equation}
which is precisely the result of Bank and Scott~\cite[Theorems~4.2 and~5.2]{BanSco89}.
In this case, the diagonal scaling becomes
\[
s_j = \left( A_{jj} \right)^{\frac{1}{2}} =
  \Biggl( \sum_{K \in \omega_j} \Abs{K} \nabla\phi_j \cdot \nabla\phi_j \Biggr)^{\frac{1}{2}} 
   \sim \Biggl(\sum_{K \in \omega_j} h_K^{d-2}\Biggr)^{\frac{1}{2}} \sim h_j^{\frac{d-2}{2}},
\]
where $h_j$ denotes the average length of the elements around the $j^{\text{th}}$ vertex.
This  scaling is equivalent to the change of basis functions 
\[
   \phi_j \to h_j^{\frac{2- d}{2}} \phi_j,
\]
which is used in~\cite[Example~2.1]{BanSco89}.
\end{scase}

\begin{scase}[Uniform meshes with respect to $\D^{-1}$]
\label{rem:D-uniform}
For a mesh that is uniform with respect to $\D^{-1}$, i.e., \emph{coefficient adaptive}, we have properties \eqref{eq:M:u:mesh} and \eqref{D-factor-2}.
Bounds \eqref{eq:conditionNew:1D}-\eqref{eq:conditionNew:1} reduce to
\begin{align*}
   &\kappa(A) 
   \leq \frac{C \left(N \Abs{\omega_{\max}}\right) }{d_{\min}} \left (\frac{N}{\sigma_h}\right )^{\frac 2 d}
      \begin{cases}
         1, & \text{for $d = 1$},\\
         1 + \ln \frac{|\bar{K}|}{\Abs{K_{\min}}}, & \text{for $d = 2$}, \\
         \left(\frac{1}{N}\sum\limits_{K\in\cT_h} \left (\frac{\abs{\bar{K}}}{\Abs{K}}\right )^{\frac{d}{2}-1}
            \right)^{\frac{2}{d}}, & \text{for $d\geq 3$},
   \end{cases}
   \\
   &\kappa(S^{-1}AS^{-1}) 
      \leq \frac{C }{d_{\min}} \left (\frac{N}{\sigma_h}\right )^{\frac 2 d},\quad \text{for $d \geq 1$},
   \notag
\end{align*}
where $\sigma_h$ is defined in \eqref{sigma-1} and corresponds to the volume of the domain in the metric specified by $\D^{-1}$.
Thus, the condition number of the scaled stiffness matrix for a coefficient adaptive mesh has the optimal order of $\cO(N^{\frac{2}{d}})$.
\end{scase}

\begin{scase}[Aligned meshes, $d \geq 2$]
For meshes aligned with the diffusion matrix but not necessarily fully coefficient adaptive (i.e., isotropic but not uniform with respect to $\D^{-1}$) we have
\[
   Q_{ali,\D^{-1}}(K) = 1 \quad \text{but} \quad Q_{eq,\D^{-1}}(K) \neq 1 .
\]
From \eqref{D-factor-1}, bound \eqref{eq:conditionNew:1} becomes
\[
   \kappa(S^{-1} A S^{-1}) 
   \leq C \frac{N^{\frac{2}{d}}}{d_{\min}} 
      \left(\frac{1}{N} \sum_{K\in\cT_h} \det(\D_K)^{\frac{1}{2}} \right)^{\frac{2}{d}}
   \begin{cases}
      1 + \ln \frac{\abs{\bar{K}}}{\Abs{K_{\min}}}, & \text{for $d = 2$}, \\
      1, & \text{for $d \geq3$}.
   \end{cases}
\]
Aside from the term depending on $\det(\D)$, this bound is equivalent to \eqref{BS89:1}. 
Hence, the diagonal scaling almost eliminates the effects of the mesh on the condition number for $\D$-aligned meshes. 
\end{scase}

\begin{scase}[General $M$-uniform meshes]
Finally, let us consider general $M$-uniform meshes, i.e., meshes that are uniform in the metric specified by a given metric tensor $M$ which does not necessarily correspond to $\D^{-1}$.
In the context of mesh adaptation, an adaptive mesh is typically generated based on some estimate of the solution error and the associated metric tensor $M$ is solution dependent.
Thus, it is of interest to know what the impact of a given $M$ on the conditioning of the stiffness matrix is. 
Recall~\cite{Huang06} that an $M$-uniform mesh satisfies
\[
   (F_K')^{-T} (F_K')^{-1} = \left( \frac{N}{\sigma_{h,M}} \right)^{\frac{2}{d}} M_K,
\]
where $M_K$ is some average of $M$ on $K$ and $\sigma_{h,M}$ is defined as in \eqref{sigma-1} but with $\D$ replaced by $M^{-1}$.
We have
\[
   \Norm{\FDF}_2 \leq  \left (\frac{N}{\sigma_{h,M}}\right )^{\frac 2 d} \Norm{M_K\D_K}_2
\]
and therefore
\[
   \kappa(S^{-1}AS^{-1}) 
   \leq \frac{C }{d_{\min}} \left (\frac{N}{\sigma_{h,M}}\right )^{\frac 2 d} 
   \left(\sum_{K\in \cT_h} \Abs{K} \Norm{M_K \D_K}_2^\frac{d}{2}\right)^{\frac{2}{d}}.
\]
Hence, the bound on the condition number after diagonal scaling for an $M$-uniform mesh depends only on the volume-weighted average of $\Norm{M_K \D_K}_2^{{d}/{2}}$ or, asymptotically, the $L^{d/2}$ norm of $\Norm{M \D}_2$.
For many problems such as those having boundary layers and shock waves, mesh elements are typically concentrated in a small portion of the physical domain.
In that situation, we would expect that $M$ differs significantly from $\D^{-1}$ only in small regions.
As a consequence, the volume-weighted average of $\Norm{M_K \D_K}_2^{{d}/{2}}$ over the whole domain may remain small and therefore the condition number of the scaled stiffness matrix for anisotropic adaptive meshes does not necessarily increase as much as generally feared.

This effect can be observed in Examples~\ref{ex:2dGeneric} and~\ref{ex:3dGeneric}.
Figures~\ref{fig:2d:arN} and~\ref{fig:3d:ar:n12} show that the effects of anisotropic adaptation are completely neutralized by the diagonal scaling when the number of anisotropic elements is small in comparison to $N$.
\end{scase}

\section{Numerical experiments}
\label{sect:numericalExperiments}

In this section we present numerical results for a selection of one-, two-, and three-dimensional examples to illustrate our theoretical findings.

Note that all bounds on the smallest eigenvalue contain a constant $C$.
We obtain its value by calibrating the bound for $\lambda_{\min}(S^{-1}AS^{-1})$ with Delaunay (Example~\ref{ex:a1}) or uniform meshes (all other examples) through comparing the exact and estimated values.
For the largest eigenvalue we use explicit bounds \eqref{eq:lambdaMaxA} and \eqref{eq:lambdaMaxScaled}.

First, we give examples with predefined meshes to demonstrate the influence of the number and shape of mesh elements on the condition number of the stiffness matrix and to verify the improvement achieved with the diagonal scaling.
For the tests, we employ the Laplace operator (i.e. $\D = I$) and a mesh on the unit interval, square, and cube, for 1D, 2D, and 3D, respectively.

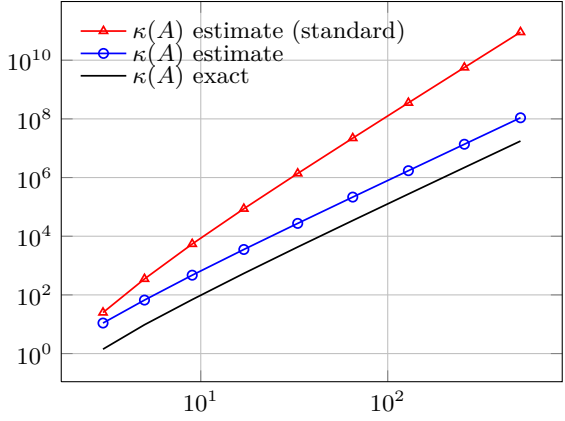
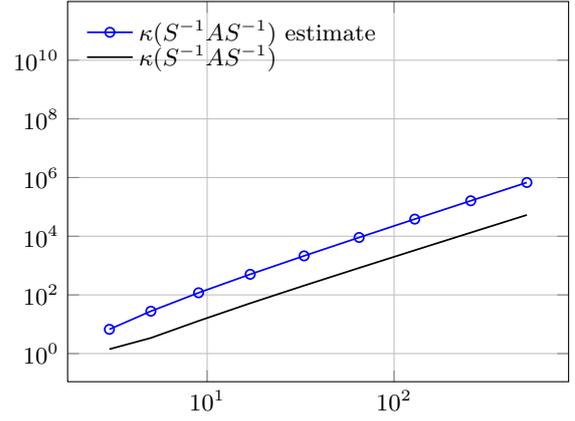
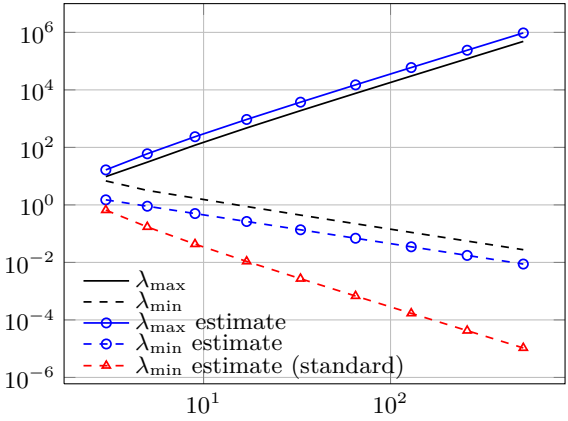
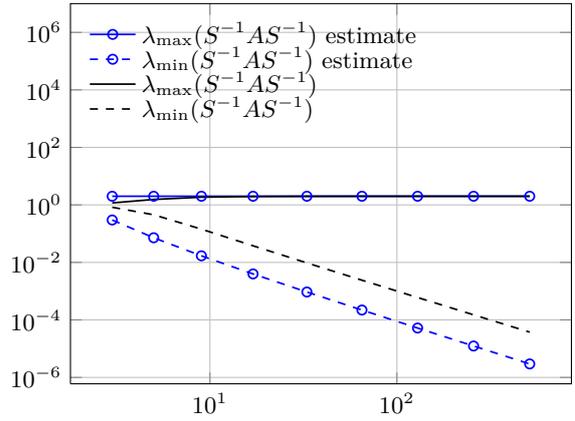
\begin{figure}[t] \centering
   \begin{subfigure}[t]{0.47\textwidth}
      \centering
      \footnotesize{
      \begin{tikzpicture}
         \begin{loglogaxis}[
            width=0.85\textwidth, height=0.65\textwidth,
            scale only axis,
            legend cell align=left,
            ymin=1.1e-01, ymax=1.0e+12,
            ytick={1.0e+00,1.0e+02,1.0e+04,1.0e+06, 1.0e+08, 1.0e+10},
            ymajorgrids,
            xmajorgrids, xminorticks=false,
            legend style={at={(0.02,0.98)},anchor=north west,align=left,fill=none,draw=none,row sep=-0.4em}
         ]
         \addplot[color=\MyColorAlert, solid, line width=\MyLineWidth, mark size=\MyMarkSize,
               mark=triangle, mark options={solid} ]
            table [x index=0, y index=3, col sep = space] {\DataPath/chebyshev-kappa.dat};
         \addlegendentry{$\kappa(A)$ estimate (standard)}
         \addplot[color=\MyColorEsti, solid, line width=\MyLineWidth, mark size=\MyMarkSize, 
               mark=o, mark options={solid} ]
            table [x index=0, y index=2, col sep = space] {\DataPath/chebyshev-kappa.dat};
         \addlegendentry{$\kappa(A)$ estimate}
         \addplot[color=\MyColorExact, solid, line width=\MyLineWidth]  
            table [x index=0, y index=1, col sep = space] {\DataPath/chebyshev-kappa.dat};
         \addlegendentry{$\kappa(A)$ exact}
           \end{loglogaxis}
      \end{tikzpicture}
      }
      \caption{Condition number.}
      \label{fig:cosCond}
   \end{subfigure}
   \quad
   \begin{subfigure}[t]{0.47\textwidth}
      \centering
      \footnotesize{
      \begin{tikzpicture}
         \begin{loglogaxis}[
            width=0.85\textwidth, height=0.65\textwidth,
            scale only axis,
            legend cell align=left,
            ymin=1.1e-01, ymax=1.0e+12,
            ytick={1.0e+00,1.0e+02,1.0e+04,1.0e+06, 1.0e+08, 1.0e+10},
            ymajorgrids,
            xmajorgrids, xminorticks=false,
            legend style={at={(0.02,0.98)},anchor=north west,align=left,fill=none,draw=none,row sep=-0.4em}
         ]
         \addplot[color=\MyColorEsti, solid, line width=\MyLineWidth, mark size=\MyMarkSize, 
               mark=o, mark options={solid} ]
            table [x index=0, y index=5, col sep = space] {\DataPath/chebyshev-kappa.dat};
         \addlegendentry{$\kappa(S^{-1}AS^{-1})$ estimate}
         \addplot[color=\MyColorExact, solid, line width=\MyLineWidth]  
            table [x index=0, y index=4, col sep = space] {\DataPath/chebyshev-kappa.dat};
         \addlegendentry{$\kappa(S^{-1}AS^{-1})$ }
          \end{loglogaxis}
      \end{tikzpicture}
      }
      \caption{Condition number after scaling.}
      \label{fig:cosCondScaled}
   \end{subfigure}
   \\[0.75em]
   \begin{subfigure}[t]{0.47\textwidth}
      \centering
      \footnotesize{
      \begin{tikzpicture}
         \begin{loglogaxis}[
            width=0.85\textwidth, height=0.65\textwidth,
            scale only axis,
            legend cell align=left,
            ymin=6.0e-07, ymax=1.0e+07,
            ytick={1.0e-06,1.0e-04,1.0e-02,1.0e+00,1.0e+02,1.0e+04,1.0e+06},
            ymajorgrids,
            xmajorgrids, xminorticks=false,
            legend style={at={(0.02,-0.01)},anchor=south west,align=left,fill=none,draw=none,row sep=-0.4em}
         ]
         \addplot[color=\MyColorExact, solid, line width=\MyLineWidth]  
            table [x index=0, y index=1, col sep = space] {\DataPath/chebyshev-lMax.dat};
         \addlegendentry{$\lambda_{\max}$}
         \addplot[color=\MyColorExact, dashed, line width=\MyLineWidth]  
            table [x index=0, y index=1, col sep = space] {\DataPath/chebyshev-lMin.dat};
         \addlegendentry{$\lambda_{\min}$}
         \addplot[color=\MyColorEsti, solid, line width=\MyLineWidth, mark size=\MyMarkSize, 
               mark=o, mark options={solid}]
            table [x index=0, y index=2, col sep = space] {\DataPath/chebyshev-lMax.dat};
         \addlegendentry{$\lambda_{\max}$ estimate}
         \addplot[color=\MyColorEsti, dashed, line width=\MyLineWidth, mark size=\MyMarkSize,
               mark=o, mark options={solid}]
            table [x index=0, y index=2, col sep = space] {\DataPath/chebyshev-lMin.dat};
         \addlegendentry{$\lambda_{\min}$ estimate}
         \addplot[color=\MyColorAlert, dashed, line width=\MyLineWidth, mark size=\MyMarkSize, mark=triangle, mark options={solid}]
            table [x index=0, y index=3, col sep = space] {\DataPath/chebyshev-lMin.dat};
         \addlegendentry{$\lambda_{\min}$ estimate (standard)}
         \end{loglogaxis}
      \end{tikzpicture}
      }
      \caption{Extreme eigenvalues.}
      \label{fig:cosEigv}
   \end{subfigure}
   \quad
   \begin{subfigure}[t]{0.47\textwidth}
      \centering
      \footnotesize{
      \begin{tikzpicture}
         \begin{loglogaxis}[
            width=0.85\textwidth, height=0.65\textwidth,
            scale only axis,
            legend cell align=left,
            ymin=6.0e-07, ymax=1.0e+07,
            ytick={1.0e-06,1.0e-04,1.0e-02,1.0e+00,1.0e+02,1.0e+04,1.0e+06},
            ymajorgrids,
            xmajorgrids, xminorticks=false,
            legend style={at={(0.02,0.98)},anchor=north west,align=left,fill=none,draw=none,row sep=-0.4em}
         ]
         \addplot[color=\MyColorEsti, solid, line width=\MyLineWidth, mark size=\MyMarkSize,
               mark=o, mark options={solid}]
            table [x index=0, y index=5, col sep = space] {\DataPath/chebyshev-lMax.dat};
         \addlegendentry{$\lambda_{\max}(S^{-1}AS^{-1})$ estimate}
         \addplot[color=\MyColorEsti, dashed, line width=\MyLineWidth, mark size=\MyMarkSize,
               mark=o, mark options={solid}]
            table [x index=0, y index=5, col sep = space] {\DataPath/chebyshev-lMin.dat};
         \addlegendentry{$\lambda_{\min}(S^{-1}AS^{-1})$ estimate}
         \addplot[color=\MyColorExact, solid, line width=\MyLineWidth]  
            table [x index=0, y index=4, col sep = space] {\DataPath/chebyshev-lMax.dat};
         \addlegendentry{$\lambda_{\max}(S^{-1}AS^{-1})$}
         \addplot[color=\MyColorExact, dashed, line width=\MyLineWidth]  
            table [x index=0, y index=4, col sep = space] {\DataPath/chebyshev-lMin.dat};
         \addlegendentry{$\lambda_{\min}(S^{-1}AS^{-1})$}
        \end{loglogaxis}
      \end{tikzpicture}
      }
      \caption{Extreme eigenvalues after scaling.}
      \label{fig:cosEigvScaled}
   \end{subfigure}
   \caption{Example~\ref{ex:1dCos}: Exact and estimated condition number and eigenvalues of
   the stiffness matrix as a function of $N$ ($d=1$).}
   \label{fig:cosExample}
\end{figure}

\begin{example}[$d=1$, $\D = I$, Chebyshev nodes]
\label{ex:1dCos}
For a simple one-dimensional example we choose  a mesh given by
Chebyshev nodes in the interval $[0,1]$,
\begin{equation}
   x_i =  \frac{1}{2} \Biggl( 1 - \cos\left( \frac{2i -1}{2 \left(N-1\right)} \pi\right) \Biggr), \quad i = 1,\dotsc, N-1.
\label{chebyshev-1}
\end{equation}
The exact condition number of the stiffness matrix and its estimates \eqref{eq:conditionNew:1D} and \eqref{eq:conditionNew:1:1} are shown in Figs.~\ref{fig:cosCond} (without scaling) and~\ref{fig:cosCondScaled} (with scaling) while those for the extreme eigenvalues and their estimates are given in Figs.~\ref{fig:cosEigv} (without scaling) and~\ref{fig:cosEigvScaled} (with scaling).

Figure~\ref{fig:cosCond} shows that the estimate \eqref{eq:conditionNew:1D} is much sharper than the standard estimate with $\lambda_{\min}(A) \propto \Abs{K_{\min}}$.
The former has the same asymptotic order as the exact value as $N$ increases, whereas the latter is too pessimistic and has a higher asymptotic order.
The difference is caused by the estimate of the smallest eigenvalue (Fig.~\ref{fig:cosEigv}).
Notice that the estimates on the largest eigenvalue are very tight, both for the scaled and the unscaled cases.

The results clearly show the benefits of diagonal scaling: the order for the condition number of the scaled stiffness matrix in Fig.~\ref{fig:cosCondScaled} is $\cO(N^2 \ln N)$, which is almost the same as for uniform meshes, whereas that without scaling in Fig.~\ref{fig:cosCond} is $\cO(N^3)$.
It can be shown analytically that the orders of the nonuniformity factors in \eqref{eq:conditionNew:1D} and \eqref{eq:conditionNew:1:1} for the Chebyshev nodes defined with \eqref{chebyshev-1} are $\cO(N)$ and $\cO(\ln N)$ and those of the corresponding condition numbers are $\cO(N^3)$ and $\cO(N^2 \ln N)$.

Thus, the numerical and theoretical results are consistent and the improvement by diagonal scaling from the maximum norm to the $L^2$ norm is significant in this example.
\end{example}

\begin{example}[$d=2$, $\D = I$, anisotropic elements in a unit square]
\label{ex:2dGeneric}
For this 2D example we use a mesh for the unit square $[0,1]\times[0,1]$ with $\cO(N^{1/2})$ skew elements, as shown in Fig.~\ref{fig:2d:exampleN}.
First, we fix the maximum aspect ratio at $125:1$ and increase $N$ to verify the dependence of the condition number on $N$ (Fig.~\ref{fig:2d:arN}).
Then, we fix $N$ at 20,000 and change the maximum aspect ratio of the mesh elements to investigate the dependence of the conditioning on the mesh shape (Fig.~\ref{fig:2d:fixedNN}).

Figure~\ref{fig:2d:arN} shows the averaging effect of the diagonal scaling: the scaling significantly reduces the condition number and, when $N$ becomes large enough, the conditioning of a scaled system is comparable to the condition number on a uniform mesh.
Moreover, the estimated value of the condition number with or without scaling has the same order as the exact value as $N$ increases.

Figure~\ref{fig:2d:fixedNN} provides a good numerical validation of \eqref{eq:conditionNew}, namely that the condition number of the unscaled stiffness matrix is linearly proportional to the largest aspect ratio\footnote{In 2D with $\D=I$, the nonuniformity term in \eqref{eq:conditionNew} is equivalent to the aspect ratio.}.
With scaling, the condition number is still increasing with an increasing aspect ratio, since the average aspect ratio (in accordance to \eqref{eq:conditionNew:1}) is also increasing.
Nevertheless, the condition number after scaling is smaller by a factor of $10$.

Figure~\ref{fig:2d:fixedNN} also shows that our estimate of the condition number with scaling has the same (linear) order as the exact value as the maximum aspect ratio increases, whereas the bounds for the unscaled case has a slightly higher order.
This indicates that the estimation can be further improved.

As for the estimates on the extreme eigenvalues, the results are mainly the same as in Example~\ref{ex:1dCos}.
For this reason, we omit them in 2D and 3D to save space.
\end{example}

\begin{figure}[p] \centering
   \begin{subfigure}[t]{0.48\textwidth}
      \centering
      \includegraphics[width=0.5\textwidth,clip]{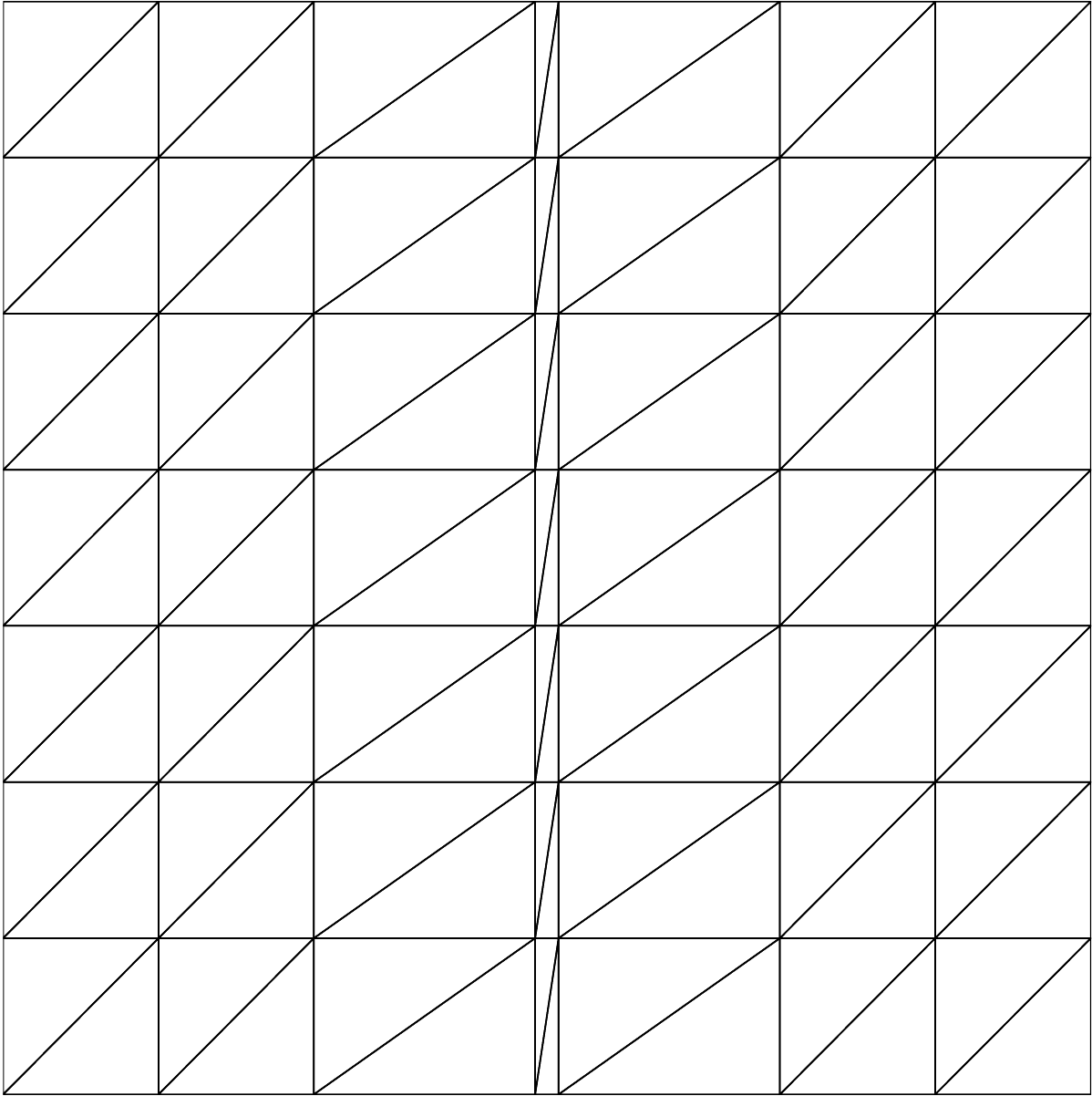}
      \caption{$\cO(N^{1/2})$ skew elements.}
      \label{fig:2d:exampleN}
   \end{subfigure}
   \begin{subfigure}[t]{0.48\textwidth}
      \centering
      \includegraphics[width=0.5\textwidth,clip]{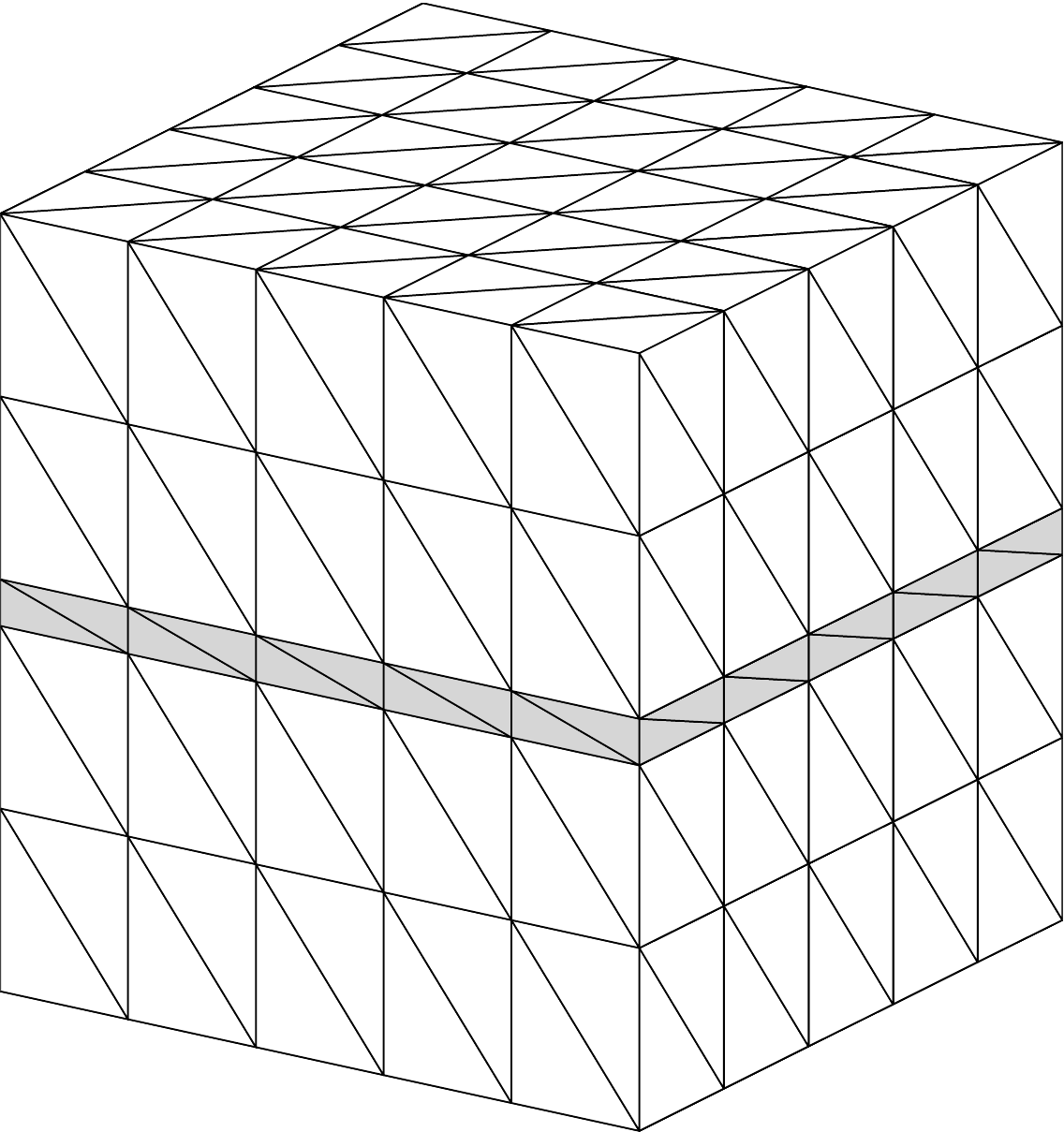}
      \caption{$\cO(N^{2/3})$ skew elements.}
      \label{fig:3d:exampleN}
   \end{subfigure}
   \caption{Predefined meshes for 
      (\subref{fig:2d:exampleN}) Example~\ref{ex:2dGeneric} and
      (\subref{fig:3d:exampleN}) Example~\ref{ex:3dGeneric}.}
   \label{fig:exampleN}
\end{figure}

\begin{figure}[p]
   \begin{subfigure}[t]{0.48\textwidth}
      \centering
      \footnotesize{ 
      \begin{tikzpicture}
      \begin{loglogaxis}[
            width=0.85\textwidth, height=0.65\textwidth,
            scale only axis,
            legend cell align=left,
            ymajorgrids, yminorticks=false, 
            xmajorgrids, xminorticks=false,
            legend style={at={(0.02,1.00)},anchor=north west,align=left,fill=none,draw=none,row sep=-0.4em}
         ]
         \addplot[color=\MyColorEsti, solid, line width=\MyLineWidth, mark size=\MyMarkSize,
               mark=o, mark options={solid}]
            table [x index=0, y index=2, col sep = space] {\DataPath/n2-kappa.dat};
         \addlegendentry{$\kappa(A)$ estimate}
         \addplot[color=\MyColorEsti, dashed, line width=\MyLineWidth, mark size=\MyMarkSize,
               mark=o, mark options={solid}]
            table [x index=0, y index=5, col sep = space] {\DataPath/n2-kappa.dat};
         \addlegendentry{$\kappa(S^{-1}AS^{-1})$ estimate}
         \addplot[color=\MyColorAlert, dotted, line width=\MyLineWidthDotted]
            table [x index=0, y index=1, col sep = space] {\DataPath/uniform2-kappa.dat};
         \addlegendentry{with uniform mesh}
         \addplot[color=\MyColorExact, solid, line width=\MyLineWidth]  
            table [x index=0, y index=1, col sep = space] {\DataPath/n2-kappa.dat};
         \addlegendentry{$\kappa(A)$}
         \addplot[color=\MyColorExact, dashed, line width=\MyLineWidth]  
            table [x index=0, y index=4, col sep = space] {\DataPath/n2-kappa.dat};
         \addlegendentry{$\kappa(S^{-1}AS^{-1})$}
           \end{loglogaxis}
      \end{tikzpicture}
      }
      \caption{Aspect ratio $125:1$, changing $N$.}
      \label{fig:2d:arN}
   \end{subfigure}
   \quad
   \begin{subfigure}[t]{0.48\textwidth}
      \centering
      \footnotesize{
      \begin{tikzpicture}
      \begin{loglogaxis}[
            width=0.85\textwidth, height=0.65\textwidth,
            scale only axis,
            legend cell align=left,
            ymajorgrids, yminorticks=false, 
            xmajorgrids, xminorticks=false,
            legend style={at={(0.02,1.00)},anchor=north west,align=left,fill=none,draw=none,row sep=-0.4em}
         ]
         \addplot[color=\MyColorEsti, solid, line width=\MyLineWidth, mark size=\MyMarkSize, 
               mark=o, mark options={solid}]
            table [x index=8, y index=2, col sep = space] {\DataPath/ar2-kappa.dat};
         \addlegendentry{$\kappa(A)$ estimate}
         \addplot[color=\MyColorEsti, dashed, line width=\MyLineWidth, mark size=\MyMarkSize, 
               mark=o, mark options={solid}]
            table [x index=8, y index=5, col sep = space] {\DataPath/ar2-kappa.dat};
         \addlegendentry{$\kappa(S^{-1}AS^{-1})$ estimate}
         \addplot[color=\MyColorExact, solid, line width=\MyLineWidth]  
            table [x index=8, y index=1, col sep = space] {\DataPath/ar2-kappa.dat};
         \addlegendentry{$\kappa(A)$}
         \addplot[color=\MyColorExact, dashed, line width=\MyLineWidth]  
            table [x index=8, y index=4, col sep = space] {\DataPath/ar2-kappa.dat};
         \addlegendentry{$\kappa(S^{-1}AS^{-1})$}
          \end{loglogaxis}
      \end{tikzpicture}
      }
      \caption{$N = 20,000$, changing aspect ratio.}
      \label{fig:2d:fixedNN}
      \end{subfigure}
      \caption{Example~\ref{ex:2dGeneric}: Condition number before and after scaling 
      for a predefined 2D mesh (Fig.~\ref{fig:2d:exampleN}) as a function of 
      (\subref{fig:2d:arN}) the number of mesh elements and
      (\subref{fig:2d:fixedNN}) the maximum element aspect ratio.}
   \label{fig:2dGeneric}
\end{figure}

\begin{figure}[p]
   \begin{subfigure}[t]{0.48\textwidth}
      \centering
      \footnotesize{ 
      \begin{tikzpicture}
      \begin{loglogaxis}[
            width=0.85\textwidth, height=0.65\textwidth,
            scale only axis,
            legend cell align=left,
            ymajorgrids, yminorticks=false, 
            xmajorgrids, xminorticks=false,
            legend style={at={(0.02,1.00)},anchor=north west,align=left,fill=none,draw=none,row sep=-0.4em}
         ]
         \addplot[color=\MyColorEsti, solid, line width=\MyLineWidth, mark size=\MyMarkSize,
               mark=o, mark options={solid}]
            table [x index=0, y index=2, col sep = space] {\DataPath/n3-kappa.dat};
         \addlegendentry{$\kappa(A)$ estimate}
         \addplot[color=\MyColorEsti, dashed, line width=\MyLineWidth, mark size=\MyMarkSize,
               mark=o, mark options={solid}]
            table [x index=0, y index=5, col sep = space] {\DataPath/n3-kappa.dat};
         \addlegendentry{$\kappa(S^{-1}AS^{-1})$ estimate}
         \addplot[color=\MyColorAlert, dotted, line width=\MyLineWidthDotted]
            table [x index=0, y index=1, col sep = space] {\DataPath/uniform3-kappa.dat};
         \addlegendentry{with uniform mesh}
         \addplot[color=\MyColorExact, solid, line width=\MyLineWidth]  
            table [x index=0, y index=1, col sep = space] {\DataPath/n3-kappa.dat};
         \addlegendentry{$\kappa(A)$}
         \addplot[color=\MyColorExact, dashed, line width=\MyLineWidth]  
            table [x index=0, y index=4, col sep = space] {\DataPath/n3-kappa.dat};
         \addlegendentry{$\kappa(S^{-1}AS^{-1})$}
           \end{loglogaxis}
      \end{tikzpicture}
      }
      \caption{Aspect ratio $25:25:1$, changing N.}
      \label{fig:3d:ar:n12}
   \end{subfigure}
   \quad
   \begin{subfigure}[t]{0.48\textwidth}
      \centering
      \footnotesize{
      \begin{tikzpicture}
      \begin{loglogaxis}[
            width=0.85\textwidth, height=0.65\textwidth,
            scale only axis,
            legend cell align=left,
            ymajorgrids, yminorticks=false, 
            xmajorgrids, xminorticks=false,
            legend style={at={(0.02,1.00)},anchor=north west,align=left,fill=none,draw=none,row sep=-0.4em}
         ]
         \addplot[color=\MyColorEsti, solid, line width=\MyLineWidth, mark size=\MyMarkSize, 
               mark=o, mark options={solid}]
            table [x index=8, y index=2, col sep = space] {\DataPath/ar3-kappa.dat};
         \addlegendentry{$\kappa(A)$ estimate}
         \addplot[color=\MyColorEsti, dashed, line width=\MyLineWidth, mark size=\MyMarkSize, 
               mark=o, mark options={solid}]
            table [x index=8, y index=5, col sep = space] {\DataPath/ar3-kappa.dat};
         \addlegendentry{$\kappa(S^{-1}AS^{-1})$ estimate}
         \addplot[color=\MyColorExact, solid, line width=\MyLineWidth]  
            table [x index=8, y index=1, col sep = space] {\DataPath/ar3-kappa.dat};
         \addlegendentry{$\kappa(A)$}
         \addplot[color=\MyColorExact, dashed, line width=\MyLineWidth]  
            table [x index=8, y index=4, col sep = space] {\DataPath/ar3-kappa.dat};
         \addlegendentry{$\kappa(S^{-1}AS^{-1})$}
          \end{loglogaxis}
      \end{tikzpicture}
      }
      \caption{$N = 29,478$, changing aspect ratio.}
      \label{fig:3d:fixedN:n12}
      \end{subfigure}
      \caption{Example~\ref{ex:3dGeneric}: Condition number before and after scaling 
         for a predefined 3D mesh (Fig.~\ref{fig:3d:exampleN}) as a function of 
         (\subref{fig:3d:ar:n12}) the number of mesh elements and 
         (\subref{fig:3d:fixedN:n12}) the maximum element aspect ratio.}
      \label{fig:3dGeneric}
\end{figure}

\begin{example}[$d=3$; anisotropic elements in a unit cube]
\label{ex:3dGeneric}

In this example, we repeat the same test setting as in Example~\ref{ex:2dGeneric}: fixed anisotropy ($25:25:1$) with increasing number of elements (Fig.~\ref{fig:3d:ar:n12}) and a fixed $N = 29,478$ paired with the changing anisotropy of the mesh (Fig.~\ref{fig:3d:fixedN:n12}).
The results shown in Fig.~\ref{fig:3dGeneric} are essentially the same as in 2D.
Since the mesh used in this example has a larger share of skew elements ($\cO(N^{-1/3})$) than the mesh used in Example~\ref{ex:2dGeneric} ($\cO(N^{-1/2})$), it is reasonable to expect that the averaging effect of diagonal scaling is less effective.
This can be seen in Fig.~\ref{fig:3dGeneric} where the exact condition numbers with and without scaling stay closer than in Fig.~\ref{fig:2dGeneric}.

Figure~\ref{fig:3dGeneric} shows that the bounds on the condition number with and without scaling have the same asymptotic order as the exact values as $N$ increases.
However, they have higher orders as the maximum aspect ratio increases for a fixed $N$.
As in the previous example, this indicates that the estimation can be further improved.
\end{example}

In the next example, we consider an adaptive finite element solution of an anisotropic diffusion problem with different meshes.
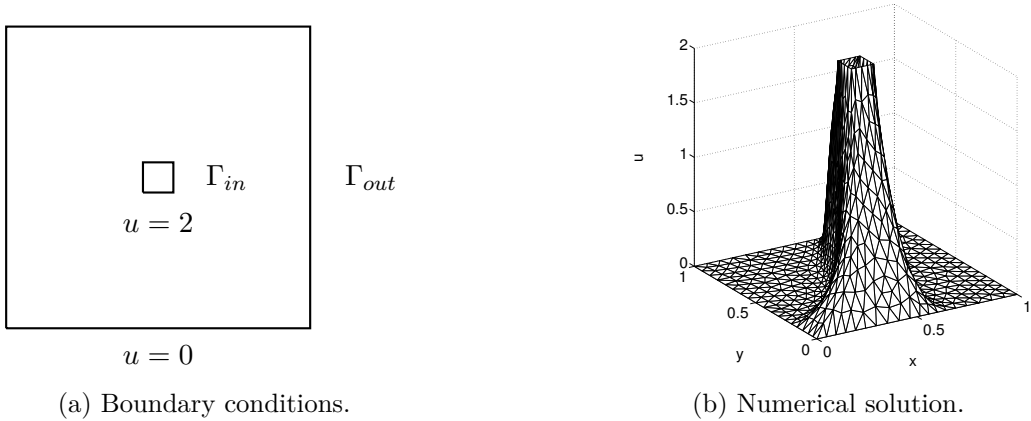
\begin{figure}[t]
   \centering
   \begin{subfigure}[b]{0.49\textwidth}
      \centering
      \begin{tikzpicture}
         \draw [thick] (0,0) -- (0,4) -- (4, 4) -- (4, 0) -- (0, 0);
         \draw [thick] (1.8,1.8) -- (1.8,2.2) -- (2.2, 2.2) -- (2.2, 1.8) -- (1.8, 1.8);
         \draw(2, -0.35) node {$ u = 0 $};
         \draw(4.8, 2) node {$ \Gamma_{out} $};
         \draw(2, 1.4) node {$ u = 2 $};
         \draw(2.9, 2) node {$ \Gamma_{in} $};
   \end{tikzpicture}
      \caption{Boundary conditions.}
      \label{fig:a1_domain}
   \end{subfigure}
   \begin{subfigure}[b]{0.49\textwidth}
      \centering
      \includegraphics[width=0.66\textwidth,clip]{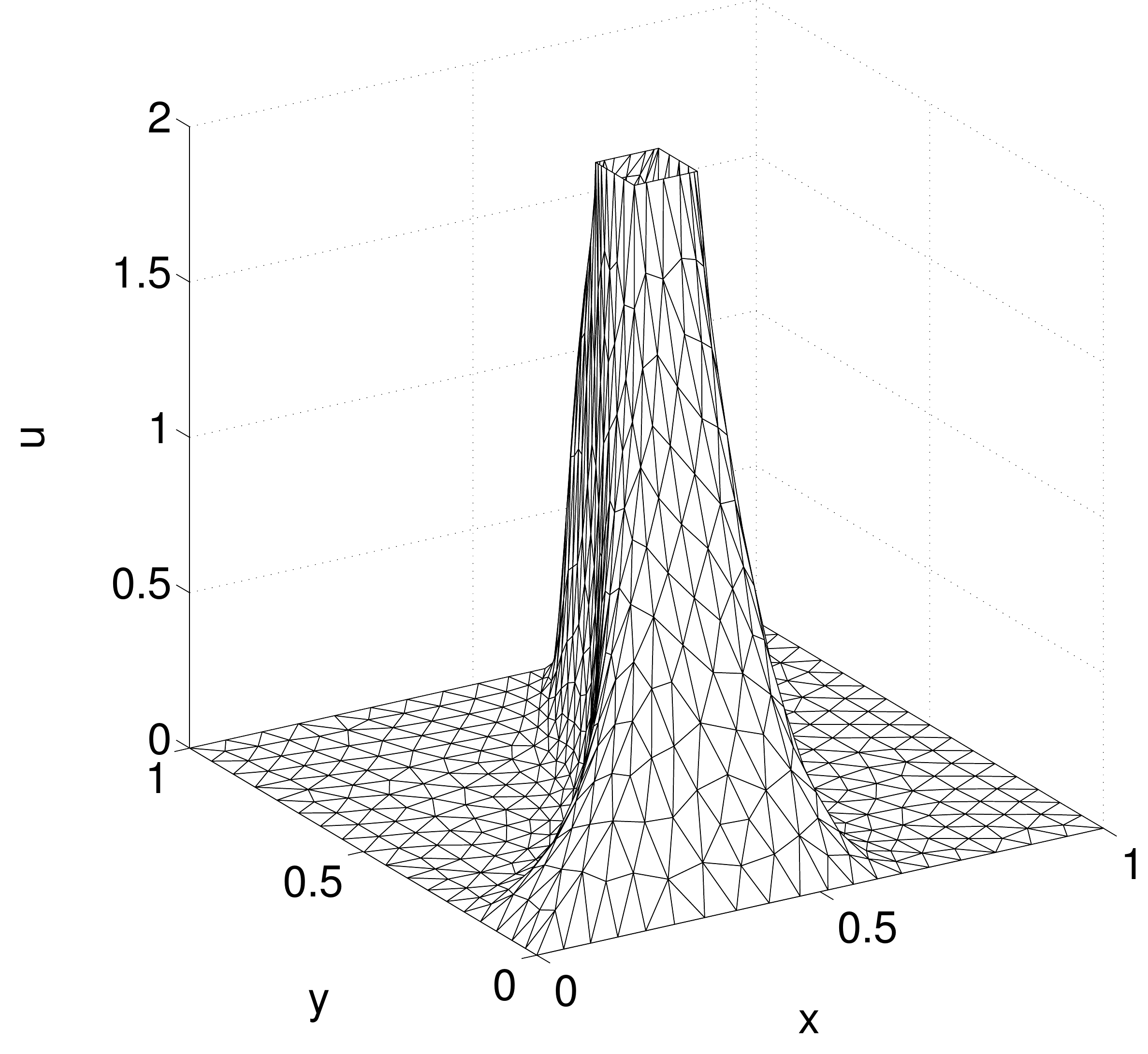}
      \caption{Numerical solution.}
      \label{fig:a1_solution}
   \end{subfigure}
   \caption{Example~\ref{ex:a1}: Boundary conditions and a numerical solution. }
   \label{fig:a1_problem}
\end{figure}

\begin{figure}[p] \centering
   \begin{subfigure}[t]{1.0\textwidth}
      \centering
      \includegraphics[height=0.2\textheight]{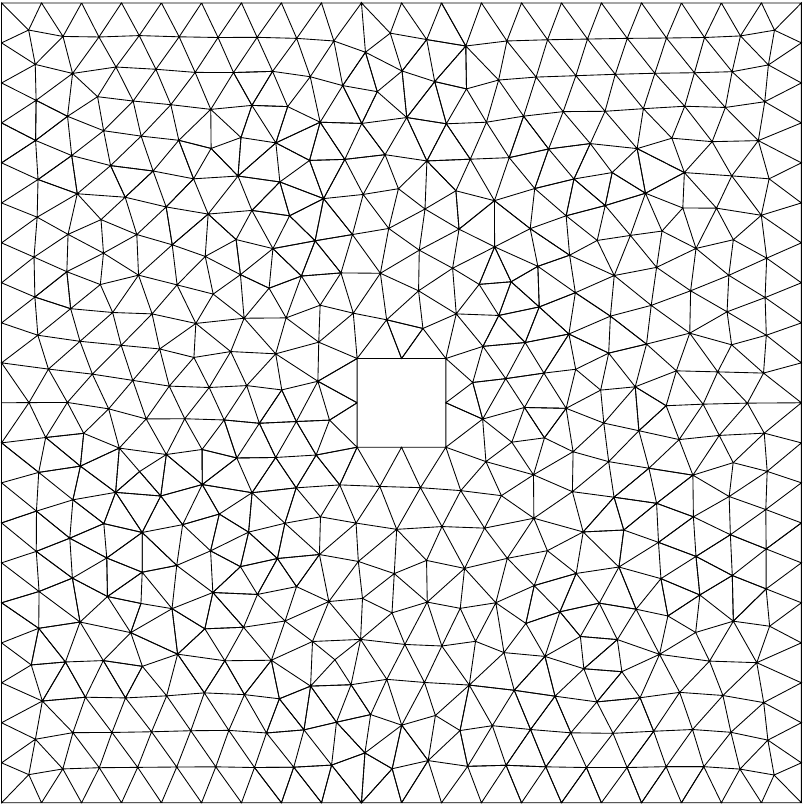}
      \qquad
      \footnotesize{ 
      \begin{tikzpicture}
      \begin{loglogaxis}[
            width=0.45\textwidth, height=0.25\textwidth,
            scale only axis,
            legend cell align=left,
            ymin=1.0e+00, ymax=5.0e+06, ytick={1.0e+01,1.0e+03,1.0e+05,1.0e+07},
            ymajorgrids, yminorticks=false, 
            xmin=2.0e+01, xmax=4.0e+05,
            xmajorgrids, xminorticks=false,
            legend style={at={(0.02,1.00)},anchor=north west,align=left,fill=none,draw=none,row sep=-0.4em}
         ]
         \addplot[color=\MyColorEsti, solid, line width=\MyLineWidth, mark size=\MyMarkSize, 
               mark=o, mark options={solid}]
            table [x index=0, y index=2, col sep = space] {\DataPath/aniso1hb-07-kappa.dat};
         \addlegendentry{$\kappa(A)$ estimate}
         \addplot[color=\MyColorEsti, dashed, line width=\MyLineWidth, mark size=\MyMarkSize, 
               mark=o, mark options={solid}]
            table [x index=0, y index=5, col sep = space] {\DataPath/aniso1hb-07-kappa.dat};
         \addlegendentry{$\kappa(S^{-1}AS^{-1})$ estimate}
         \addplot[color=\MyColorExact, solid, line width=\MyLineWidth]  
            table [x index=0, y index=1, col sep = space] {\DataPath/aniso1hb-07-kappa.dat};
         \addlegendentry{$\kappa(A)$}
         \addplot[color=\MyColorExact, dashed, line width=\MyLineWidth]  
            table [x index=0, y index=4, col sep = space] {\DataPath/aniso1hb-07-kappa.dat};
         \addlegendentry{$\kappa(S^{-1}AS^{-1})$}
          \end{loglogaxis}
      \end{tikzpicture}
      }
      \caption{Quasi-uniform (Delaunay) mesh, $M = I$.}
      \label{fig:a1:7}
   \end{subfigure}
   \\[0.7em]
   \begin{subfigure}[t]{1.0\textwidth}
      \centering
      \includegraphics[height=0.2\textheight]{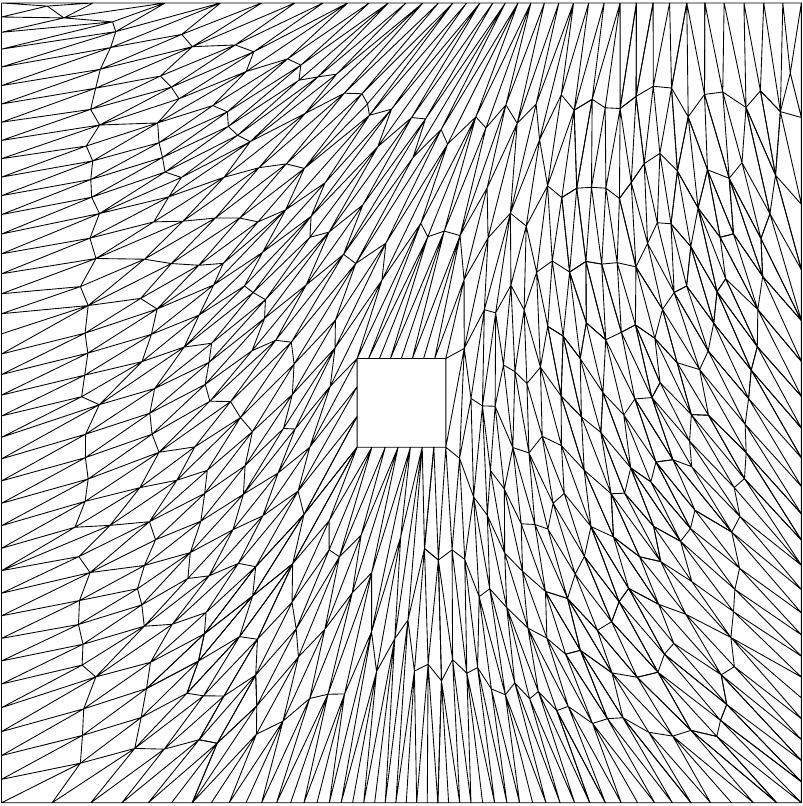}
      \qquad
      \footnotesize{ 
      \begin{tikzpicture}
      \begin{loglogaxis}[
            width=0.45\textwidth, height=0.25\textwidth,
            scale only axis,
            legend cell align=left,
            ymin=1.0e+00, ymax=5.0e+06, ytick={1.0e+01,1.0e+03,1.0e+05,1.0e+07},
            ymajorgrids, yminorticks=false, 
            xmin=2.0e+01, xmax=4.0e+05,
            xmajorgrids, xminorticks=false,
            legend style={at={(0.02,1.00)},anchor=north west,align=left,fill=none,draw=none,row sep=-0.4em}
         ]
         \addplot[color=\MyColorEsti, solid, line width=\MyLineWidth, mark size=\MyMarkSize,
               mark=o, mark options={solid}]
            table [x index=0, y index=2, col sep = space] {\DataPath/aniso1hb-10-kappa.dat};
         \addlegendentry{$\kappa(A)$ estimate}
         \addplot[color=\MyColorEsti, dashed, line width=\MyLineWidth,mark size=\MyMarkSize,
               mark=o, mark options={solid}]
            table [x index=0, y index=5, col sep = space] {\DataPath/aniso1hb-10-kappa.dat};
         \addlegendentry{$\kappa(S^{-1}AS^{-1})$ estimate}
         \addplot[color=\MyColorAlert, dotted, line width=\MyLineWidthDotted]
            table [x index=0, y index=1, col sep = space]  {\DataPath/aniso1hb-07-kappa.dat};
         \addlegendentry{$\kappa(A)$ Delaunay}
         \addplot[color=\MyColorExact, solid, line width=\MyLineWidth]  
            table [x index=0, y index=1, col sep = space] {\DataPath/aniso1hb-10-kappa.dat};
         \addlegendentry{$\kappa(A)$}
         \addplot[color=\MyColorExact, dashed, line width=\MyLineWidth]  
            table [x index=0, y index=4, col sep = space] {\DataPath/aniso1hb-10-kappa.dat};
         \addlegendentry{$\kappa(S^{-1}AS^{-1})$}
         \end{loglogaxis}
      \end{tikzpicture}
      }
      \caption{Coefficient-adaptive mesh, $M = \D^{-1}$.}
      \label{fig:a1:10}
   \end{subfigure}
   \\[0.7em]
   \begin{subfigure}[t]{1.0\textwidth}
      \centering
      \includegraphics[height=0.2\textheight]{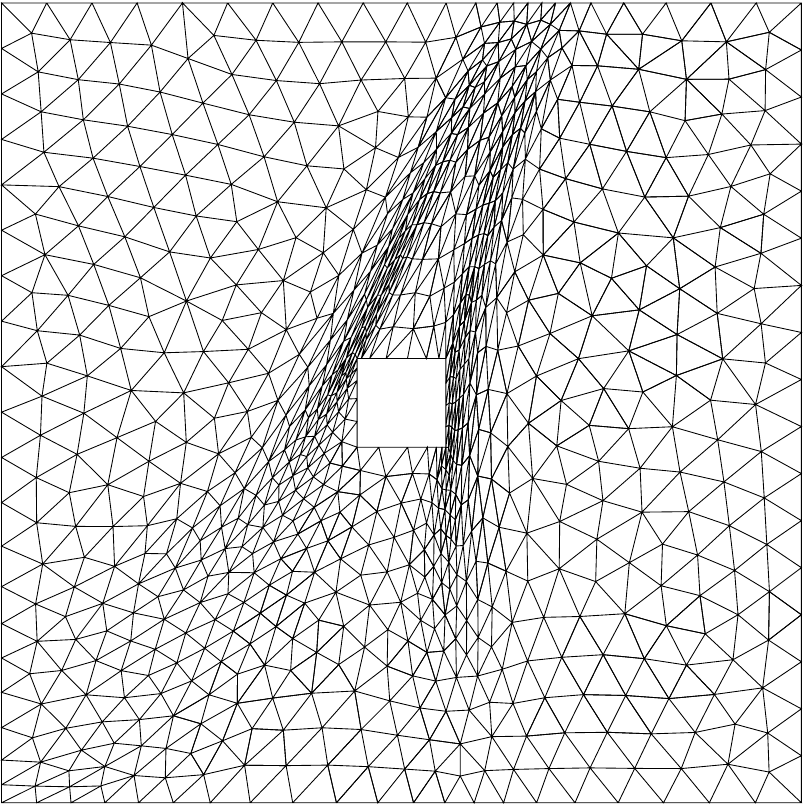}
      \qquad
      \footnotesize{ 
      \begin{tikzpicture}
      \begin{loglogaxis}[
            width=0.45\textwidth, height=0.25\textwidth,
            scale only axis,
            legend cell align=left,
            ymin=1.0e+00, ymax=5.0e+06, ytick={1.0e+01,1.0e+03,1.0e+05,1.0e+07},
            ymajorgrids, yminorticks=false, 
            xmin=2.0e+01, xmax=4.0e+05,
            xmajorgrids, xminorticks=false,
            legend style={at={(0.02,1.00)},anchor=north west,align=left,fill=none,draw=none,row sep=-0.4em}
         ]
         \addplot[color=\MyColorEsti, solid, line width=\MyLineWidth, mark size=\MyMarkSize,
               mark=o, mark options={solid}]
            table [x index=0, y index=2, col sep = space] {\DataPath/aniso1hb-05-kappa.dat};
         \addlegendentry{$\kappa(A)$ estimate}
         \addplot[color=\MyColorEsti, dashed, line width=\MyLineWidth, mark size=\MyMarkSize,
               mark=o, mark options={solid}]
            table [x index=0, y index=5, col sep = space] {\DataPath/aniso1hb-05-kappa.dat};
         \addlegendentry{$\kappa(S^{-1}AS^{-1})$ estimate}
         \addplot[color=\MyColorAlert, dotted, line width=\MyLineWidthDotted]
            table [x index=0, y index=1, col sep = space]  {\DataPath/aniso1hb-07-kappa.dat};
         \addlegendentry{$\kappa(A)$ Delaunay}
         \addplot[color=\MyColorExact, solid, line width=\MyLineWidth]  
            table [x index=0, y index=1, col sep = space] {\DataPath/aniso1hb-05-kappa.dat};
         \addlegendentry{$\kappa(A)$}
         \addplot[color=\MyColorExact, dashed, line width=\MyLineWidth] 
            table [x index=0, y index=4, col sep = space] {\DataPath/aniso1hb-05-kappa.dat};
         \addlegendentry{$\kappa(S^{-1}AS^{-1})$}
          \end{loglogaxis}
      \end{tikzpicture}
      }
      \caption{Solution-adaptive mesh, $M = M(u_h)$.}
      \label{fig:a1:5}
   \end{subfigure}
   \\[0.7em]
   \begin{subfigure}[t]{1.0\textwidth}
      \centering
      \includegraphics[height=0.2\textheight]{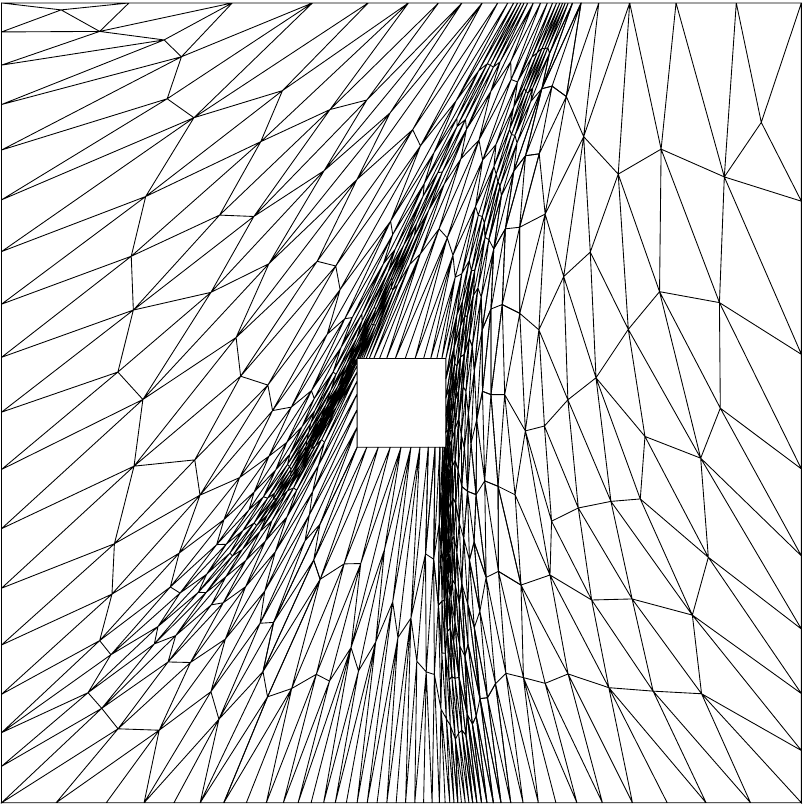}
      \qquad
      \footnotesize{ 
      \begin{tikzpicture}
      \begin{loglogaxis}[
            width=0.45\textwidth, height=0.25\textwidth,
            scale only axis,
            legend cell align=left,
            ymin=2.0e+00, ymax=2.0e+06, 
            ymin=1.0e+00, ymax=5.0e+06, ytick={1.0e+01,1.0e+03,1.0e+05,1.0e+07},
            ymajorgrids, yminorticks=false, 
            xmin=2.0e+01, xmax=4.0e+05,
            xmajorgrids, xminorticks=false,
            legend style={at={(0.02,1.00)},anchor=north west,align=left,fill=none,draw=none,row sep=-0.4em}
         ]
         \addplot[color=\MyColorEsti, solid, line width=\MyLineWidth, mark size=\MyMarkSize,
               mark=o, mark options={solid}]
            table [x index=0, y index=2, col sep = space] {\DataPath/aniso1hb-11-kappa.dat};
         \addlegendentry{$\kappa(A)$ estimate}
         \addplot[color=\MyColorEsti, dashed, line width=\MyLineWidth,mark size=\MyMarkSize,
               mark=o, mark options={solid}]
            table [x index=0, y index=5, col sep = space] {\DataPath/aniso1hb-11-kappa.dat};
         \addlegendentry{$\kappa(S^{-1}AS^{-1})$ estimate}
         \addplot[color=\MyColorAlert, dotted, line width=\MyLineWidthDotted]
            table [x index=0, y index=1, col sep = space]  {\DataPath/aniso1hb-07-kappa.dat};
         \addlegendentry{$\kappa(A)$ Delaunay}
         \addplot[color=\MyColorExact, solid, line width=\MyLineWidth]  
            table [x index=0, y index=1, col sep = space] {\DataPath/aniso1hb-11-kappa.dat};
         \addlegendentry{$\kappa(A)$}
         \addplot[color=\MyColorExact, dashed, line width=\MyLineWidth]
            table [x index=0, y index=4, col sep = space] {\DataPath/aniso1hb-11-kappa.dat};
         \addlegendentry{$\kappa(S^{-1}AS^{-1})$}
          \end{loglogaxis}
      \end{tikzpicture}
      }
      \caption{Coefficient- and solution-adaptive mesh,  $M = \theta(u_h) \D^{-1}$.}
      \label{fig:a1:11}
   \end{subfigure}
   \caption{Example~\ref{ex:a1}: condition number in dependence of $N$.}
   \label{fig:a1}
\end{figure}

\begin{example}[$d=2$, adaptive anisotropic meshes]
\label{ex:a1}

Consider an anisotropic diffusion problem studied in~\cite{Kamens11,LiHua10}.
It takes the form of BVP \eqref{eq:bvp1} but with a non-homogeneous Dirichlet boundary condition.
The domain and its outer and inner boundaries $\Gamma_{out}$ and $\Gamma_{in}$ are shown in Fig.~\ref{fig:a1_domain}.
The coefficients of the BVP, the right-hand side and the boundary data are given by
\begin{align}
\D &= 
   \begin{bmatrix}
      \cos\psi   & - \sin\psi \\ 
      \sin \psi &  \cos\psi
   \end{bmatrix}
   \begin{bmatrix}
     1000   & 0 \\ 
     0      & 1
   \end{bmatrix}
   \begin{bmatrix}
      \cos\psi    & \sin\psi \\
      - \sin\psi  & \cos\psi
   \end{bmatrix},
   \quad \psi = \pi \sin x \cos y,
   \label{eq:exam-2} \\
   f &= 0 \text{ in $\Omega=\left[0,1\right]^2\setminus\left[ \frac{4}{9},\frac{5}{9} \right]^2$},
   \quad g = 0 \text{ on $\Gamma_{out}$}, 
   \quad g = 2 \text{ on $\Gamma_{in}$}.
   \notag
\end{align}

We employ an adaptive finite element algorithm from~\cite{HuKaLa10,Kamens11} to compute the numerical solution and adaptive meshes.
The algorithm utilizes the $M$-uniform mesh approach, i.e., meshes are generated as quasi-uniform in a given metric $M$. 
For the mesh generation we use the \emph{bidimensional anisotropic mesh generator}~\cite{bamg}.

A Delaunay mesh\textemdash{}our first example (Fig.~\ref{fig:a1:7})\textemdash{}is $M$-uniform (or $M$-quasi-uniform) with respect to $M = I$.
The second mesh (Fig.~\ref{fig:a1:10}) is purely coefficient-adaptive and is defined as an $M$-uniform mesh with respect to $\D^{-1}$, i.e., $M = \D^{-1}$.
The third mesh (Fig.~\ref{fig:a1:5}) is a purely solution-adaptive mesh where $M = M(u_h)$ depends on the numerical solution $u_h$ (or, more precisely, on the hierarchical basis error estimate $e_h$).
The fourth mesh (Fig.~\ref{fig:a1:11}) represents a combination of adaptation to both the solution and the coefficients of the problem and the metric is defined as  $M = \theta(e_h) \D^{-1}$, where $\theta(e_h)$ is a scalar function depending on the error estimator $e_h$.
With such choice the shape of mesh elements is determined by the diffusion matrix while the size is controlled by the estimate of the solution error.

From Fig.~\ref{fig:a1} we can see that the smallest condition number among all four meshes is with the purely coefficient-adaptive mesh (Fig.~\ref{fig:a1:10}), which is consistent with Special Case~\ref{rem:D-uniform}.
The conditioning is better than in the case of a quasi-uniform mesh (Fig.~\ref{fig:a1:7}), confirming the observation that, depending on the problem, a quasi-uniform mesh is not necessarily the best mesh from the conditioning point of view.
For both cases, diagonal scaling does not improve the condition number significantly.
This is expected since both meshes are almost volume-uniform.
To explain why the mesh in Fig.~\ref{fig:a1:10} is (almost) volume-uniform, we recall from \eqref{Q-eq} and \eqref{eq:M:u:mesh} that an $M$-uniform mesh with respect to $\D^{-1}$ satisfies
\[
   \Abs{K} \sim \sqrt{\det(\D_K)} \quad \forall K \in \cT_h.
\]
The diffusion matrix $\D$ in \eqref{eq:exam-2} satisfies $\det(\D) = 1000$.
Thus, $\Abs{K} = const$.

The largest condition number is in the case of the purely solution-adaptive mesh (Fig.~\ref{fig:a1:5}).
This is because the mesh is not volume-uniform and its elements are not aligned with $\D^{-1}$.
Since the mesh is far from being uniform in size, scaling will have a significant impact, as it can be verified in Fig.~\ref{fig:a1:5}: the condition number after the scaling is even smaller than the condition number with Delauney meshes. 

Conditioning with a mesh that is both coefficient- and solution-adaptive (Fig.~\ref{fig:a1:11}) is not as good as in the case of the purely coefficient-adaptive mesh but better than in the case of the purely adaptive and Delaunay meshes.

In all four cases we observe that the developed estimates for the condition number of the stiffness matrix  are reasonably tight and have the same order as the exact values as $N$ increases for both unscaled and scaled cases.

\end{example}

\section{Summary and conclusions}
\label{sect:summary}

\subsection*{Mass matrix}
Our new estimate \eqref{eq:massMatrixBound:1} of the condition number of the Galerkin mass matrix is tight within a factor of $(d+2)$ from \emph{both above and below} for \emph{any mesh} with \emph{no assumptions} on mesh regularity or topology.
It this sense, it is optimal and truly anisotropic.

\subsection*{Stiffness matrix}
Lemma~\ref{lem:lambdaMax} provides an estimate of the largest eigenvalue of the stiffness matrix which is simple to compute and is tight within a factor of $(d+1)$ from \emph{both above and below} for \emph{any mesh}.
This is in contrast to many existing estimates which are proportional to the maximal number of elements meeting at a mesh point.

New bounds \eqref{eq:lambdaMin}--\eqref{eq:lambdaMinScaled} on the smallest eigenvalue and \eqref{eq:conditionNew:1D}--\eqref{eq:conditionNew:1} on the condition number of the stiffness matrix are a significant improvement in comparison to the previously available estimates.

First, the new bounds show that the conditioning of the stiffness matrix with an arbitrary (anisotropic) mesh is much better than generally assumed, especially for $d=1$ and $d=2$.

Second, the new bounds are truly anisotropic and valid for any mesh since no assumptions on the mesh regularity were made.

Third, bounds \eqref{eq:conditionNew:1D} and \eqref{eq:conditionNew} reveal what affects the conditioning.
There are three factors.
The first (base) factor $C N^{\frac{2}{d}}$ describes the direct dependence of the condition number on the \emph{number of mesh elements} and corresponds to the condition number for the Laplace operator on a uniform mesh.
The second factor describes the effects of the \emph{mesh $\D^{-1}$-nonuniformity}, i.e., the interplay between the shape and size of mesh elements and the coefficients of the BVP.\@
It is $\cO(1)$ for a coefficient-adaptive mesh, i.e., a mesh satisfying \eqref{eq:M:u:mesh}.
The third factor measures how the \emph{mesh volume-nonuniformity} further affects the condition number.
It has no effect in 1D, a minimal one in 2D, and a substantial effect in 3D and higher dimensions.
This means that even if the mesh is coefficient-adaptive and the second factor is $\cO(1)$, the mesh volume-nonuniformity can still have a significant impact on the condition number for $d \geq 3$.

Fourth, a simple diagonal scaling, such as the Jacobi preconditioning, can significantly improve the conditioning.
Bound \eqref{eq:conditionNew:1} for the condition number after scaling does not contain the factor for the mesh volume-nonuniformity.
As a consequence, for a coefficient-adaptive mesh, this bound reduces to the base factor $C N^{\frac{2}{d}}$.
In this sense, diagonal scaling eliminates the effects of the mesh volume-nonuniformity.
It can also significantly reduce the effects of the mesh nonuniformity with respect to $\D^{-1}$: the influence reduces essentially from the maximum norm to the $L^{\max \{1,\frac{d}{2}\}}$ norm of $\norm{\FDF}_2$.

Moreover, for a preconditioner that is invariant to diagonal scaling it follows that the condition number of the preconditioned stiffness matrix is typically smaller than $\kappa(S^{-1} A S^{-1})$ which in turn has a much lower bound than $\kappa (A)$ (cf. \eqref{eq:conditionNew} and \eqref{eq:conditionNew:1}). 
For example, consider  an incomplete Cholesky decomposition of $A$, 
\[
   A = L L^T + E.
\]
It follows that
\[
   S^{-1} A S^{-1} = \left(S^{-1} L\right) {\left(S^{-1}L\right)}^T + S^{-1}ES^{-1}
\]
is actually an incomplete Cholesky decomposition of $S^{-1} A S^{-1}$ since $S^{-1} L$ has the same sparsity pattern as $L$.
Then from the identity
\[
   L^{-1} A L^{-T}  =  {\left(S^{-1}L\right)}^{-1} \left(S^{-1} A S^{-1}\right) {\left(S^{-1}L\right)}^{-T}
\] 
we see that the preconditioned matrix of $A$ with preconditioner $L$ is equivalent to the preconditioned matrix of $S^{-1} A S^{-1}$ with preconditioner $S^{-1}L$.
As a result, the performance of the preconditioning technique on $A$ is the same as that on $S^{-1} A S^{-1}$ which has a much smaller condition number than $A$.
Although there is no estimate yet on the condition number of the preconditioned system, the above observation may provide a partial explanation for the good performance of  ILU preconditioners with anisotropic meshes observed in~\cite{Huang05a}.

Numerical experiments (Figs.~\ref{fig:2d:arN} and~\ref{fig:3d:ar:n12}) indicate that although the new bounds have the same order as the exact value as the number of elements increases, they may have higher asymptotic orders than the exact value as the element aspect ratio increases.
These may deserve further investigations.

Finally, we would like to point out that although the study in this paper has been done specifically for the linear finite element discretization, the approach can be generalized for higher order finite elements without major modifications.

\subsubsection*{Acknowledgement}

L. K. is very thankful to Jonathan R. Shewchuk for a fruitful discussion at the ICIAM 2011 and
for pointing out valuable references. 
The authors are very grateful to the anonymous referees for their comments and suggestions which helped to significantly improve the quality of this paper.

This work was supported in part by
   the DFG (Germany) under grants KA~3215/1--1 and KA~3215/2--1 
and
   the NSF (U.S.A.) under grants DMS–0712935 and DMS–1115118.

\bibliographystyle{abbrv}
\bibliography{KaHuXu11}

\end{document}